\documentclass[12pt, leqno,twoside]{article}
\usepackage{amssymb}
\usepackage{amsmath}
\numberwithin{equation}{section}
\usepackage{amsthm}
\usepackage{graphicx}
\usepackage[pagebackref,colorlinks,citecolor=blue,linkcolor=blue]{hyperref}
\usepackage{cite}

\usepackage{todonotes}
\usepackage{amsmath}
\usepackage{amsfonts}
\usepackage{amssymb}
\usepackage{amsmath,bbm,amssymb,amsxtra}
\usepackage{mathrsfs}
\usepackage{enumerate}
\usepackage{caption}
\allowdisplaybreaks
\usepackage{epsfig}
\usepackage{ae}
\usepackage{epstopdf}
\def\vint{\mathop{\mathchoice%
         {\setbox0\hbox{$\displaystyle\intop$}\kern 0.22\wd0%
          \vcenter{\hrule width 0.6\wd0}\kern -0.82\wd0}%
         {\setbox0\hbox{$\textstyle\intop$}\kern 0.2\wd0%
          \vcenter{\hrule width 0.6\wd0}\kern -0.8\wd0}%
         {\setbox0\hbox{$\scriptstyle\intop$}\kern 0.2\wd0%
          \vcenter{\hrule width 0.6\wd0}\kern -0.8\wd0}%
         {\setbox0\hbox{$\scriptscriptstyle\intop$}\kern 0.2\wd0%
          \vcenter{\hrule width 0.6\wd0}\kern -0.8\wd0}}%
         \mathopen{}\int}

\newtheorem{thm}{Theorem}[section]

\newtheorem{cor}[thm]{Corollary}
\newtheorem{lem}[thm]{Lemma}
\newtheorem{prop}[thm]{Proposition}

\newtheorem{claim}{Claim}[section]
\newtheorem{subclaim}{Subclaim}
\newtheorem{conj}[equation]{Conjecture}
\newtheorem{case}{Case}[section]
\newtheorem*{mysolution}{Solution}
\newtheorem{step}{Step}[section]
\theoremstyle{definition}
\newtheorem{defn}[thm]{Definition}
\newtheorem{example}[thm]{Example}
\newtheorem{prob}[equation]{Problem}
\newtheorem{ques}[equation]{Question}
\newtheorem{rem}{Remark}[section]
\newtheorem{rems}{Remarks}[section]
\newcounter {own}
\def\theown {\thesection       .\arabic{own}}

\newenvironment{pf}[1][]{%
 \vskip 3mm
 \noindent
 \ifthenelse{\equal{#1}{}}%
  {{\slshape Proof. }}%
  {{\slshape #1.} }%
 }%
{\qed\bigskip}

\newcounter{alphabet}

\makeatletter

\newcommand{\RNum}[1]{\uppercase\expandafter{\romannumeral #1\relax}}

\def\be{\begin{equation}}
\def\ee{\end{equation}}

\newcommand{\ben}{\begin{enumerate}}
\newcommand{\een}{\end{enumerate}}

\newcommand{\blem}{\begin{lem}}
\newcommand{\elem}{\end{lem}}
\newcommand{\bthm}{\begin{thm}}
\newcommand{\ethm}{\end{thm}}
\newcommand{\bcor}{\begin{cor}}
\newcommand{\ecor}{\end{cor}}
\newcommand{\beg}{\begin{examp}}
\newcommand{\eeg}{\end{examp}}
\newcommand{\begs}{\begin{examples}}
\newcommand{\eegs}{\end{examples}}
\newcommand{\bdefe}{\begin{defn}}
\newcommand{\edefe}{\end{defn}}
\newcommand{\bprob}{\begin{prob}}
\newcommand{\eprob}{\end{prob}}
\newcommand{\bques}{\begin{ques}}
\newcommand{\eques}{\end{ques}}
\newcommand{\bei}{\begin{itemize}}
\newcommand{\eei}{\end{itemize}}
\newcommand{\bcl}{\begin{claim}}
\newcommand{\ecl}{\end{claim}}
\newcommand{\bscl}{\begin{subclaim}}
\newcommand{\escl}{\end{subclaim}}
\newcommand{\bca}{\begin{case}}
\newcommand{\eca}{\end{case}}
\newcommand{\bstep}{\begin{step}}
\newcommand{\estep}{\end{step}}
\newcommand{\bsol}{\begin{mysolution}}
\newcommand{\esol}{\end{mysolution}}
\newcommand{\bcon}{\begin{conj}}
\newcommand{\econ}{\end{conj}}
\newcommand{\bcons}{\begin{conjs}}
\newcommand{\econs}{\end{conjs}}
\newcommand{\bprop}{\begin{prop}}
\newcommand{\eprop}{\end{prop}}
\newcommand{\br}{\begin{rem}}
\newcommand{\er}{\end{rem}}
\newcommand{\brs}{\begin{rems}}
\newcommand{\ers}{\end{rems}}
\newcommand{\bo}{\begin{obser}}
\newcommand{\eo}{\end{obser}}
\newcommand{\bos}{\begin{obsers}}
\newcommand{\eos}{\end{obsers}}

\newcommand{\bpf}{\begin{pf}}
\newcommand{\epf}{\end{pf}}

\newcommand{\ba}{\begin{array}}
\newcommand{\ea}{\end{array}}
\newcommand{\beq}{\begin{eqnarray}}
\newcommand{\beqq}{\begin{eqnarray*}}
\newcommand{\eeq}{\end{eqnarray}}
\newcommand{\eeqq}{\end{eqnarray*}}

\begin{document}
\title{\Large\bf Characterizations for the existence of traces of first-order Sobolev spaces on hyperbolic fillings
 \footnotetext{\hspace{-0.35cm}
 $2010$ {\it Mathematics Subject classfication}: 46E35, 31E05
 \endgraf{{\it Key words and phases}: Characterization, Trace operator, Trace function, Hyperbolic filling, first-order Sobolev space}
}
}
\author{Manzi Huang and Zhihao Xu}
\date{ }
\maketitle
\begin{abstract}
In this paper, we study the existence of traces for Sobolev spaces on the hyperbolic filling $X$ of a compact metric space $Z$ equipped with a doubling measure. Given a suitable metric on $X$, we can regard $Z$ as the boundary of $X$. After equipping $X$ with a weighted measure $\mu$ via the measure on $Z$ and the Euclidean arc length, we give characterizations for the existence of traces for first-order Sobolev spaces.
\end{abstract}

\section{Introduction}\label{sec-1}

Let $(Z, d_{Z}, \nu)$ be a compact metric space equipped with a doubling measure $\nu$ and $0<\text{diam}Z<1$. Then a hyperbolic filling of the underlying space $Z$ is a simplicial graph  with vertex set $V$ and edge set $E$. It carries a natural path metric obtained by identifying each edge with a copy of the unit interval. It is known that each hyperbolic filling is Gromov hyperbolic; see \cite[Chapter 6]{BuSc}. The technique of hyperbolic fillings of compact metric spaces equipped with doubling measures was introduced in Bj\"{o}rn-Bj\"{o}rn-Shanmugalingam \cite{BBS}, Bonk-Saksman \cite{BS} and  Bourdon-Pajot \cite{BP03}. In these papers, the constructions of hyperbolic fillings are different. We refer the interested readers to \cite[Section 3]{BBS} for discussions on their differences.  In this paper, we adopt the construction in \cite{BBS}; see Section \ref{sec-2} for details.

For a given  hyperbolic filling $X$ of the underlying space $Z$, we exploit the metric $d$ defined in \eqref{metric-def} below. Let $\overline{X}$ be the completion of $X$ with respect to $d$, and let $\partial X=\overline{X}\backslash X$, the uniformized boundary of $X$.
By \cite[Proposition 4.4]{BBS}, we know that the uniformized boundary $\partial X$ and $Z$ are biLipschitz equivalent, and thus, in the following, we identify $\partial X$ and $Z$ via the biLipschitz correspondence.
On $X$, we also use the measure $\mu$ defined in \eqref{-measure} below, as a result, we get a metric measure space $(X, d, \mu)$.

The motivation of this paper comes from the study of trace existence on regular trees in \cite{KNW}. It was shown that the existence of traces for first-order Sobolev spaces on  regular trees can be characterized by a parameter related to certain weighted distance and certain weighted measure defined on regular trees. Note that by \cite[Theorem 7.1]{BBS}, any regular tree can be regarded as a hyperbolic filling of its uniformalized boundary. Naturally, one will ask if there are any results similar to those in \cite{KNW} for general hyperbolic fillings.
The main purpose of this paper is to discuss this question. To state our results, we need to introduce some definitions and notations.

For any $\xi\in Z$, let $[v_0, \xi)$  denote a geodesic ray from the root $v_0$ of $X$ to $\xi$; see Definition \ref{geodesic-ray} for the precise definition of geodesic rays on $X$. Given a measurable function $u$ on $X$, we say that its {\it trace $($function$)$}  $\mathscr Tu$ exists if for $\nu$-almost every $\xi\in Z$,
\begin{equation}\label{equ-wetrf}
\mathscr Tu(\xi):=\lim_{[v_{0},\xi)\ni x\rightarrow\xi}u(x)
\end{equation}
exists and is independent of the choice of the geodesic rays from $v_{0}$ to $\xi$. Here the independence of the choice of the geodesic rays from $v_{0}$ to $\xi$ is necessary and important since unlike the regular tree case,  in hyperbolic fillings, the geodesic ray from a root to a boundary point may be not unique.

Obviously, it follows from the definition that if $\mathscr T$ exists, then it must be linear.

For convenience, if the limit $\lim_{[v_{0},\xi)\ni x\rightarrow\xi}u(x)$ does not exist, then we say that {\it $u$ has no limit along $[v_{0},\xi)$}.

In order to characterize the existence of traces for first-order Sobolev spaces, we introduce two parameters $R_{p,\rho}$ and $\mathcal{R}_{p,\rho}$ for $1\leq p<\infty$ as follows. For a Borel function $\rho:[0,\infty)\rightarrow(0,\infty)$ satisfying $\rho\in L_{\text{loc}}^{1}([0,\infty))$, let
\begin{equation}\label{Rp-1}
R_{p,\rho}=
\left\{\begin{array}{cl}
\int_{0}^{\infty}e^{-\frac{\epsilon p}{p-1}t}\rho(t)^{\frac{1}{1-p}}dt, &\ \text{if} \ \ p>1, \\
\|e^{-\epsilon t}\rho(t)^{-1}\|_{L^{\infty}([0,\infty))},&\ \text{if} \ \ p=1,
\end{array}\right.
\end{equation}
where $\epsilon=\log \alpha$ and $\alpha>1$ is one of the construction parameters of $X$; see Section \ref{sec-2} for details.

When $\mu(X)<\infty$, let $\mathcal{R}_{p,\rho}=R_{p,\rho}$, and when $\mu(X)=\infty$, let
\begin{equation}\label{Rp-2}
\mathcal{R}_{p,\rho}=
\left\{\begin{array}{cl}
\sup_{k\geq1}\int_{O_{k}}e^{-\frac{\epsilon p}{p-1}t}\rho(t)^{\frac{1}{1-p}}dt, &\,\text{if} \,\ p>1, \\
R_{1,\rho},&\,\text{if} \,\ p=1,
\end{array}\right.
\end{equation}
where $\{O_{k}\}=\{[t_{k},t_{k+1})\}_{k=1}^{\infty}$ is a sequence of subintervals of $[0,\infty)$ with $\bigcup_{k}O_{k}=[0,\infty)$, $\int_{O_{k}}\rho(t)dt=1$ for each $k$, and $O_{k}\cap O_{k'}=\emptyset$ for any $k\not= k'$.

We remark that if $\mu(X)=\infty$, then by Remark  \ref{rem-2.2} below, $\int_{0}^{\infty}\rho(t)dt=\infty$. This fact guarantees the existence of the needed sequence $\{O_{k}\}$, and so, the parameter $\mathcal{R}_{p,\rho}$ is well-defined since the assumption of $\rho$ being positive ensures that $\{O_{k}\}$ is unique. Also, it is easy to know that $R_{p,\rho}<\infty$ implies $\mathcal{R}_{p,\rho}<\infty$, while the converse does not hold true (cf. \cite{KNW}).

For $1\leq p<\infty$, let $\mathcal F_p$ denote the set of all Borel functions $\rho:[0,\infty)\rightarrow(0,\infty)$ satisfying $\rho\in L_{\text{loc}}^{1}([0,\infty))$ and
\begin{equation}\label{addition-assumption}
\left\{\begin{array}{cc}
e^{-\epsilon pt}\rho(t)^{-1}\in L^{1/(p-1)}_{\text{loc}}([0,\infty)),&\  \text{if}\ \  p>1,\\
 e^{-\epsilon t}\rho(t)^{-1}\in L^{\infty}_{\text{loc}}([0,\infty)),&\  \text{if} \ \  p=1.
 \end{array}\right.
 \end{equation}

Let $N^{1,p}(X):=N^{1,p}(X, d, \mu)$, $1\leq p<\infty$, the first-order Sobolev space, and let $\dot N^{1,p}(X):= \dot N^{1,p}(X, d, \mu)$, the homogeneous first-order Sobolev space  (see Definitions \ref{def-2.5} and \ref{def-2.6} below for their precise definitions).
 Our results concerning the existence of traces are as follows.

\begin{thm}\label{thm-1.1}
Let $1\leq p<\infty$ and $\rho\in \mathcal F_p$. Then the following statements hold true.

$(i)$  The parameter $\mathcal{R}_{p,\rho}<\infty$ if and only if the trace function $\mathscr Tu$ exists for every $u\in N^{1,p}(X)$.

$(ii)$  The parameter $R_{p,\rho}<\infty$ if and only if the trace function $\mathscr Tu$ exists for every $u\in \dot N^{1,p}(X)$.
\end{thm}

\begin{rem}
$(a)$ If $\mu(X)<\infty$, then the condition $\mathcal{R}_{p,\rho}=R_{p,\rho}<\infty$ guarantees that for every $u\in N^{1,p}(X)$, $\mathscr Tu$ exists, and $\mathscr T: N^{1,p}(X)\rightarrow L^{p}(Z,\nu)$
is a bounded operator, where $L^{p}(Z,\nu)$ is the $L^p$-space on $(Z, d_Z, \nu)$; see Lemma \ref{lem-3.1} below.
If $\mu(X)=\infty$, then the condition $\mathcal{R}_{p,\rho}<\infty$ implies that $\mathscr Tu(\xi)=0$ for $\nu$-almost every $\xi\in Z$; see Lemma \ref{lem-3.3} below.

$(b)$  The condition $R_{p,\rho}<\infty$ ensures that for every $u\in\dot N^{1,p}(X)$, $\mathscr Tu$ exists, and $\mathscr T:\dot N^{1,p}(X)\rightarrow L^{p}(Z,\nu)$
 is a bounded operator; see Lemma \ref{lem-3.1}  below.
\end{rem}

In \cite{BBS}, Bj\"{o}rn-Bj\"{o}rn-Shanmugalingam proved that for each $1\leq p<\infty$, the trace space of any Newtonian-Sobolev space, i.e., any first-order Sobolev space, $N^{1, p}(X)$ is the Besov space $B_{p, p}^{\theta}(Z)$ ($0<\theta<1$) when the underlying space $Z$ is compact and doubling, and when the hyperbolic filling $X$ is equipped with a measure
defined as in \eqref{-measure} with $\rho(t)=e^{-\epsilon p(1-\theta) t}$; see \cite[Theorem 1.1]{BBS} for details. Later, Butler \cite{BC} extended the extension and trace results in \cite{BBS} by weakening the hypothesis ``compactness" on $Z$ to the one ``completeness".
Notice that the trace functions $\tilde u$ defined in \cite{BBS} are different from the ones $\mathscr T u$ in this paper (see \cite{BBS} or Section \ref{sec-4} for the definition of $\tilde u$).  As the second purpose of this paper, we discuss the relationships between the definition of trace functions
$\mathscr Tu$ given in this paper and the one of $\widetilde{u}$ introduced in \cite{BBS}.
Our result is as follows.

\begin{thm}\label{thm-imply}
Let $u$ be a measurable function defined on $X$. Then the existence of $\mathscr Tu$ implies the existence of $\widetilde{u}$. If $\mathscr Tu$ exists, then for $\nu$-almost every $\xi\in Z$, $\mathscr Tu(\xi)=\widetilde{u}(\xi)$.
\end{thm}

In the setting of Bj\"{o}rn-Bj\"{o}rn-Shanmugalingam \cite{BBS}, elementary computations show that both the parameters $\mu(X)$ and  $\mathcal{R}_{p,\rho}$ are finite.
Then we see from Theorems \ref{thm-1.1} and \ref{thm-imply} that the trace function $\widetilde{u}$ in \cite{BBS} exists for every $u\in N^{1,p}(X)$. The following example demonstrates that the opposite implication in Theorem \ref{thm-imply} is invalid.

\begin{example}\label{exm-1.1}
Let $1\leq p<\infty$ and $\rho(t)=e^{-t^{2}-\epsilon pt}$. Then there exists a function $u\in N^{1,p}(X)$ such that the trace $\widetilde{u}$ exists, however, the trace
$\mathscr Tu$ does not exist.
\end{example}

Our work is highly inspired by the recent trace results on regular trees, hyperbolic fillings and metric measure spaces. In addition to the aforementioned references, we also refer the interested readers to \cite{BBGS, KW, KNW2, LLW, ML, WA,SS17} for more discussions in this direction.

The paper is organized as follows. In Section \ref{sec-2}, some necessary terminologies will be introduced and some auxiliary results will be proved. Theorem \ref{thm-1.1} will be shown in Section \ref{sec-3}, and Section \ref{sec-4} will be devoted to the proof of Theorem \ref{thm-imply} and the construction of Example \ref{exm-1.1}.

\section{Preliminaries}\label{sec-2}
Throughout this paper, the letter $C$ (sometimes with a subscript) will denote a positive constant that usually depend only on the given parameters of the spaces and may change at different occurrences; if $C$ depends on $a,$ $b,$ $\ldots$, then we write $C = C(a,b,\ldots).$ The notation $A\lesssim B$ (resp. $A \gtrsim B$) means that there is a constant $C\geq 1$ such that $A \leq C \cdot B$ (resp. $A \geq C \cdot B).$
If there is a constant $C\geq 1$ such that $C^{-1} \cdot A\leq B\leq C \cdot A$, then we write $A\approx B$, and $C$ is called a {\it comparison constant}.

For a metric space $(Z, d_{Z})$, the open ball $\mathbb{B}_{Z}(x, r)$ in $Z$ with center $x \in Z$ and radius $r>0$ is
$$\mathbb{B}_{Z}(x, r)=\{y\in Z: d_{Z}(y,x)<r\},$$
and for a constant $\lambda>0$, the notation $\lambda\mathbb{B}_{Z}(x, r)$ denotes the following open ball:
$$\lambda\mathbb{B}_{Z}(x, r)=\{y\in Z: d_{Z}(y,x)<\lambda r\}.$$
A metric space $(Z, d_{Z})$ is called {\it doubling} if there exists a constant $C\geq 1$ such that every ball in $Z$ with radius $r>0$ can be covered by at most $C$ balls with radii $\frac{1}{2}r$.

If $(Z, \mu)$ is a measure space, for every function $u\in L^1_{\rm loc}(X, \mu)$ and every measurable subset $A\subset X$, let $\vint_Aud\mu$ stand for the integral average $\frac{1}{\mu(A)}\int_Aud\mu$, i.e.,
$$\vint_Aud\mu=\frac{1}{\mu(A)}\int_Aud\mu.$$
The measure $\mu$ is said to be {\it doubling} if there exists a constant $C_\mu\geq1$ such that for all balls $B$ in $Z$,
\begin{equation}\label{eq-doub}
0<\mu(2B)\leq C_\mu\mu(B)<\infty.
\end{equation}

It is well known that a metric space equipped with a doubling measure is doubling (cf. \cite[Section 10.13]{H}).

We construct the hyperbolic filling as follows (see also \cite{BBS}):
Suppose that $(Z, d_Z)$ is a bounded metric space with $0<\text{diam}Z<1$. Let $\alpha,$ $\tau>1$, $z_0\in Z$ and $A_{0}=\{z_{0}\}$. Obviously, $\mathbb{B}_{Z}(z_{0},1)=Z$.
By Zorn$^{\prime}$s lemma or the Hausdorff maximality principle, for each positive integer $n$, we can choose a maximal $\alpha^{-n}$-separated set $A_{n}\subset Z$ such that $A_{n}\subset A_{m}$ provided $m> n\geq0$. Here a set $A\subset Z$ is called {\it $\alpha^{-n}$-separated} if for any distinct $z, z^{\prime}\in A_{n}$, $d_{Z}(z,z^{\prime})\geq\alpha^{-n}$. Then it follows that the balls $\mathbb{B}_{Z}(z,\frac{1}{2}\alpha^{-n})$ are pairwise disjoint for different elements $z$ in $A_{n}$. Since $A_{n}$ is maximal, we see that the balls $\mathbb{B}_{Z}(z,\alpha^{-n})$ for $z\in A_{n}$ cover $Z$. Here and from now on, $n$ and $m$ always denote nonnegative integers.

Next, we define the {\it vertex set}
$$
V=\bigcup_{n\geq0} V_{n},
$$ where $V_{0}=\{v_{0}\}$ with $v_{0}=(z_{0}, 0)$, and for any $n\geq 1$,
$\ V_{n}=\{(x, n):x\in A_{n}\}$.

By the correspondence between the points in $A_{n}$ and the points in $V_{n}$, we set $B_{v}:=\mathbb{B}_{Z}(x, \alpha^{-n})$ for $v=(x, n)\in V_{n}$.

Given two different vertices $(x, n), (y, m)$ (i.e., $x\not=y$ or $n\not= m$), we say that $(x, n)$ is a {\it neighbor} of $(y, m)$, denoted by $(x, n)\sim(y, m)$,
if and only if $|n-m|\leq1$ and
$$
\tau^{1-|n-m|} \mathbb{B}_{Z}(x,\alpha^{-n})\cap \tau^{1-|n-m|} \mathbb{B}_{Z}(y,\alpha^{-m})\neq\emptyset.
$$

Define the {\it hyperbolic filling} of $Z$, denoted by $X$, to be the graph formed by the vertex set $V$ together with the above neighbor relation (edges). In particular, we say that $v_{0}$ is the {\it root} of the hyperbolic filling $X$. Also, we say that $\alpha$  and $\tau$ are {\it construction parameters} of $X$.

We consider $X$ to be a metric graph where the edges are unit intervals.
The graph distance between two points $x, y \in X$, denoted by $|x-y|$, is the length of the shortest curve connecting them. Since $X$ is a metric graph, it is easy to see that these shortest curves exist, and thus, $X$ is geodesic. An edge $(x, n)\sim(y, m)$ is {\it horizontal} if $m=n$, and {\it vertical} if $m=n\pm1$. A geodesic is called {\it vertical} if it consists of vertical edges.

In general, each vertex has at least one neighbor. The following proposition shows that if
the metric space is doubling, then the degree is well-controlled. Here, the degree of a vertex is the
number of neighbors it has.

\begin{prop} $($cf. \cite[Proposition 4.5]{BBS}$)$\label{prop-2.1}
Assume that $(Z,d_Z)$ is a doubling metric space. Then the hyperbolic filling $X$ has uniformly bounded degree, that is, there exists a constant $K=C(\alpha, \tau, Z)\geq1$ such that for any $v\in V$,
$$
1\leq\#\{w: w\sim v\}\leq K,
$$
where $\#\{w: w\sim v\}$ is the cardinality of the set $\{w: w\sim v\}$.
\end{prop}

In the rest of this paper, we assume that $(Z, d_{Z}, \nu)$ is a compact metric space equipped with a doubling measure $\nu$ and $0<\text{diam}Z<1$, and that $X$ is a hyperbolic filling of $Z$, as constructed above, with construction parameters $\alpha, \tau$.

On $X$, we define a metric $d$  by setting
$$
ds=e^{-\epsilon|x|}d|x|,
$$
where $\epsilon=\log\alpha$, $|x|:=|x-v_{0}|$, and $d|x|$ denotes a measure which appoints the Lebesgue measure of every edge in $X$ to be $1$. Then for any pair of points $y$ and $z$ in $X$, the distance between them is
\begin{equation}\label{metric-def}
d(y,z)=\inf_{\gamma}\int_{\gamma}ds=\inf_{\gamma}\int_{\gamma}e^{-\epsilon|x|}d|x|,
\end{equation}
where the infimum is taken over all curves in $X$ joining $y$ and $z$.

We lift up the boundary measure $\nu$ on $Z$ to a measure $\mu$ on $X$ as follows. For any $x\in X$, there is either a vertex $v\in V$ such that $x=v$ or an edge $[v_1, v_2]$ such that $x\in [v_1, v_2]\backslash \{v_1, v_2\}$. Define
\begin{equation}\label{-measure}
d\mu(x)=
\left\{\begin{array}{cl}
2\rho(|x|)\nu(B_{v})d|x|, &\text{if } x=v, \\
\rho(|x|)\big(\nu(B_{v_1})+\nu(B_{v_2})\big)d|x| , &\text{if } x\in [v_1, v_2]\backslash \{v_1, v_2\},
\end{array}\right.
 \end{equation}
 where    $\rho:[0,\infty)\rightarrow(0,\infty)$ is a Borel function satisfying $\rho\in L_{\text{loc}}^{1}([0,\infty))$.
For any $x\in X$, let $v_x$ be a vertex in $X$ such that $$|x-v_x|=\min_{v\in V}|x-v|.$$
It is possible that there is some $x\in X$ such that the vertex $v_x$ may be not unique. Since $\nu$ is doubling,
\begin{equation}\label{eq-add-zw}
d\mu(x)\approx \rho(|x|)\nu(B_{v_x})d|x|
\end{equation}
with a comparison constant depending on $C_\nu, \alpha, \tau$, where $C_\nu\geq1$ is the doubling constant satisfying \eqref{eq-doub}. Also, we have
\begin{equation}\label{ds-du}
ds\approx\frac{e^{-\epsilon |x|}}{\rho(|x|)\nu(B_{v_x})}d\mu
\end{equation}
with a comparison constant depending on $C_\nu, \alpha, \tau$.

Given a boundary point $\xi\in Z$, let $$
A_{n}(\xi)=A_{n}\cap \mathbb{B}_{Z}(\xi, \alpha^{-n}) \ \ \ \ \text{and} \ \ \ V_{n}(\xi)=\{(x,n)\in V_{n}:x\in A_{n}(\xi)\}.
$$

Obviously, $A_n(\xi)\not=\emptyset$, and for any nonnegative integer $n$, $v_{n+1}(\xi)\sim v_{n}(\xi)$ since $\xi\in B_{v_{n}(\xi)}\cap B_{v_{n+1}(\xi)}$.

\begin{defn}[Geodesic ray]\label{geodesic-ray}
For any $\xi\in Z$, we say that $[v_0, \xi)$ is a geodesic ray from $v_0$ to  $\xi$ if
$$[v_{0},\xi):=\bigcup_{n\geq0}[v_{n}(\xi),v_{n+1}(\xi)],$$
where $v_{n}(\xi)\in V_{n}(\xi)$ and $v_{n+1}(\xi)\in V_{n+1}(\xi)$.
\end{defn}

Note that for any $\xi\in Z$, the geodesic ray from   the root  $v_0$ to  $\xi$ may be not unique since the vertex set $V_n(\xi)$ may have more than one elements.

\begin{defn}
Let $\Gamma$ be a family of nonconstant rectifiable curves in $X$ and $F(\Gamma)$ the family of all Borel measurable functions $\varrho: X\rightarrow[0,\infty]$ such that for every $\gamma\in\Gamma$,
$$
\int_{\gamma}\varrho ds\geq 1.
$$
For $1\leq p<\infty$, we define the $p$-modulus of the family $\Gamma$ as
$$
\text{Mod}_{p}(\Gamma)=\inf_{\varrho\in F(\Gamma)}\int_{X}\varrho^{p}d\mu.
$$
\end{defn}

The following result is an analogue of \cite[Proposition 2.2]{KNW} for hyperbolic fillings.

\begin{prop}\label{Prop-2.3}
Let $1\leq p<\infty$, and let $\rho:[0,\infty)\rightarrow(0,\infty)$ is a locally integrable Borel function. Then the following are equivalent.

$(1)$ The function $e^{-\epsilon t}\rho(t)^{-\frac{1}{p}}$ belongs to $L^{p^{\prime}}_{\text{loc}}([0,\infty))$, where $p'=\frac{p}{p-1}$ if $p>1$, and $p'=\infty$ if $p=1$.

$(2)$ For every rectifiable curve $\gamma$ in $X$, $\text{Mod}_{p}(\{\gamma\})>0$.
\end{prop}

\begin{proof}
$(1)\Rightarrow(2)$. Assume that $\gamma$ is a rectifiable curve in $X$. For any Borel measurable function $\phi\in F(\{\gamma\})$, we have $\int_{\gamma}\phi ds\geq1$. By the monotone convergence theorem, we may assume that $\int_{\gamma^{\prime}}\phi ds\geq\frac{1}{2}$ for a subcurve $\gamma^{\prime}$ of $\gamma$ with $\gamma^{\prime}\subset\{x\in X:|x|\leq N\}$ for some $N\in \mathbb{N}$.

Since $\nu$ is doubling and $Z$ is bounded, there exists a constant $C(C_\nu, N, \alpha)>0$ such that for all $x\in X$ with $|x|\leq N$,
$$
0<C(C_\nu, N, \alpha)\nu(Z)\leq \nu(B_{v_x})\leq \nu(Z)< \infty,
$$ where $C_\nu$ is from \eqref{eq-doub}.

For $p>1$, it follows from the H\"{o}lder inequality, \eqref{eq-add-zw} and \eqref{ds-du}
that
\begin{align*}
\int_{\gamma^{\prime}}\phi ds
\leqslant& C(C_\nu, \alpha, \tau)\int_{\gamma^{\prime}}\phi(x)\frac{e^{-\epsilon |x|}}{\rho(|x|)\nu(B_{v_x})}d\mu(x)\\
\leqslant&C(C_\nu, \alpha, \tau)\bigg(\int_{\gamma^{\prime}} \frac{e^{-\frac{\epsilon p|x|}{p-1}}}{\rho(|x|)^{\frac{p}{p-1}}\nu(B_{v_x})^{\frac{p}{p-1}}} d\mu(x)\bigg)^{\frac{p-1}{p}}\bigg(\int_{\gamma^{\prime}}\phi(x)^{p} d\mu(x)\bigg)^{\frac{1}{p}}\\
\leqslant&C(C_\nu, N, \alpha, \tau)\bigg(\int_{\gamma^{\prime}}e^{-\frac{\epsilon p|x|}{p-1}}\rho(|x|)^{-\frac{1}{p-1}}d|x|\bigg)^{\frac{p-1}{p}}\bigg(\int_{\gamma^{\prime}}\phi(x)^{p} d\mu(x)\bigg)^{\frac{1}{p}}.
\end{align*}

Since $V_0=\{v_0\}$, we know from Proposition \ref{prop-2.1} that for all $n\geq 0$, $$\#V_n\leq K^{n},$$ which, together with the assumption $e^{-\epsilon t}\rho(t)^{-\frac{1}{p}}\in L^{p^{\prime}}_{\text{loc}}([0,\infty))$, gives
\begin{align*}
\int_{\gamma^{\prime}} e^{-\frac{\epsilon p|x|}{p-1}}\rho(|x|)^{-\frac{1}{p-1}} d|x|
\leqslant& \int_{\{x\in X:|x|\leq N\}} e^{-\frac{\epsilon p|x|}{p-1}}\rho(|x|)^{-\frac{1}{p-1}} d|x|\\
\leqslant& K^{2N}\int_{0}^{N}e^{-\frac{\epsilon p}{p-1}t}\rho(t)^{-\frac{1}{p-1}}dt<\infty.
\end{align*}
Substituting into the above formula, we obtain that
$$
\int_{\gamma^{\prime}}\phi ds
\leqslant C(C_\nu, N, \alpha, \tau, p, \epsilon, \rho, K)\bigg(\int_{\gamma^{\prime}}\phi(x)^{p} d\mu(x)\bigg)^{\frac{1}{p}}.
$$
Then we get
\begin{align*}
\int_{X}\phi^{p} d\mu&\geq\int_{\gamma^{\prime}}\phi^{p} d\mu\geq C(C_\nu, N, \alpha, \tau, p, \epsilon, \rho, K)\bigg(\int_{\gamma^{\prime}}\phi ds\bigg)^{p}\\
&\geq\frac{1}{2^{p}}C(C_\nu, N, \alpha, \tau, p, \epsilon, \rho, K)>0.
\end{align*}

For $p=1$, it follows from \eqref{ds-du} and the assumption $e^{-\epsilon t}\rho(t)^{-1}\in L^{\infty}_{\text{loc}}([0,\infty))$ that
$$
\int_{\gamma^{\prime}}\phi ds\leq C(C_\nu, N, \alpha, \tau)\|e^{-\epsilon t}\rho(t)^{-1}\|_{L^{\infty}_{\text{loc}}([0,\infty))}\int_{\gamma^{\prime}}\phi d\mu.
$$
Hence we have
$$
\int_{X}\phi d\mu\geq\int_{\gamma^{\prime}}\phi d\mu\geq C(C_\nu, N, \alpha, \tau, \epsilon, \rho)\int_{\gamma^{\prime}}\phi ds\geq\frac{1}{2}C(C_\nu, N, \alpha, \tau, \epsilon, \rho)>0.
$$
These ensure that Mod$_{p}(\{\gamma\})>0$ for every rectifiable curve $\gamma$, which completes the check of the implication $(1)\Rightarrow(2)$.

$(2)\Rightarrow(1)$. We prove this implication by contradiction. Assume that for every rectifiable curve $\gamma$ in $X$, $\text{Mod}_{p}(\{\gamma\})>0$, but $e^{-\epsilon t}\rho(t)^{-\frac{1}{p}}\notin L^{p^{\prime}}_{\text{loc}}([0,\infty))$. To reach a contradiction, it suffices to find a rectifiable curve $\gamma$ such that $\text{Mod}_{p}(\{\gamma\})=0$. By the assumption, there exist two constants $0\leq a<b<\infty$ such that
$$
\|e^{-\epsilon t}\rho(t)^{-\frac{1}{p}}\|_{L^{p^{\prime}}([a,b])}=\infty.
$$

Let $x_{a}, y_{b}\in X$ be such that $|x_{a}|=a$, $|y_{b}|=b$. Let $\gamma$ denote a vertical geodesic joining $x_{a}$ and $y_{b}$ in $X$. Then we claim that $${\rm Mod}_{p}(\{\gamma\})=0.$$
To prove this claim, by \cite[Theorem 5.5]{HAJ}, we only need to find a Borel measurable function $0\leq\phi\in L^{p}(X, \mu)$ such that
$$
\int_{\gamma}\phi ds=\infty.
$$

For $p>1$, we have
$$
\int_{a}^{b}e^{-\frac{\epsilon p}{p-1}t}\rho(t)^{-\frac{1}{p-1}}dt=\infty.
$$
Since both $e^{-\epsilon t}$ and $\rho(t)$ are Borel measurable functions, there is a sequence of pairwise disjoint Borel subsets $\{I_{k}:I_{k}\subset[a, b]\}_{k=1}^{\infty}$ such that
$$
2^{k}<\int_{I_{k}}e^{-\frac{\epsilon p}{p-1}t}\rho(t)^{\frac{1}{1-p}}dt<\infty.
$$
We define a function $\phi:X\rightarrow\mathbb{R}$ by setting
$$
\phi(x)=h(|x|)\chi_{\gamma}(x),
$$
where $\chi_{\gamma}$ denotes the characteristic function of $\gamma$, and
$$
h(t)=\sum_{k=1}^{\infty}\frac{e^{-\frac{\epsilon t}{p-1}}\rho(t)^{\frac{1}{1-p}}}{\int_{I_{k}}e^{-\frac{\epsilon p}{p-1}t}\rho(t)^{\frac{1}{1-p}}dt}\chi_{I_{k}}(t).
$$
Since $I_{k}$ is Borel for each $k$, we know that $\phi$ is a Borel measurable function, and thus, the pairwise disjoint property of $\{I_{k}\}$ implies
\begin{align*}
\int_{X}\phi^{p}d\mu=&\int_{\gamma}h(|x|)^{p}\rho(|x|)\nu(B_{v_x})d|x|\leqslant\nu(Z)\int_{a}^{b}h(t)^{p}\rho(t)dt\\
\lesssim &\sum_{k=1}^{\infty}\int_{I_{k}}\frac{e^{-\frac{\epsilon p}{p-1}t}\rho(t)^{\frac{1}{1-p}}}{\big(\int_{I_{k}}e^{-\frac{\epsilon p}{p-1}t}\rho(t)^{\frac{1}{1-p}}dt\big)^{p}}dt\\
< & \sum_{k=1}^{\infty}2^{-k(p-1)}<\infty.
\end{align*} This guarantees that $\phi\in L^{p}(X, \mu)$.
Also, we can get
$$
\int_{\gamma}\phi ds=\int_{a}^{b}h(t)e^{-\epsilon t}dt=\sum_{k=1}^{\infty}1=\infty.
$$

For the remaining case, that is, $p=1$, by assumptions, we have
$$
\|e^{-\epsilon t}\rho(t)^{-1}\|_{L^{\infty}([a, b])}=\infty.
$$
There exists an infinite sequence $\{G_{k_{n}}\}_{k_{n}\in\mathbb{N}}$ defined by
$$
G_{k_{n}}=\big\{t\in[a, b]: 2^{k_{n+1}}\geq e^{-\epsilon t}\rho(t)^{-1}>2^{k_{n}}\big\}
$$
such that
$
m(G_{k_{n}})>0
$
for each $k_{n}\in\mathbb{N}$, where $m(\cdot)$ denotes the $1$-dimension Lebesgue measure.
We define a function $\phi:X\rightarrow\mathbb{R}$ by setting
$$
\phi(x)=\widehat{h}(|x|)\chi_{\gamma}(x),
$$
where
$$
\widehat{h}(t)=\sum_{k_{n}\in\mathbb{N}}\frac{1}{\int_{G_{k_{n}}}e^{-\epsilon t}dt}\chi_{G_{k_{n}}}(t).
$$
Since the subsets $\{G_{k_{n}}\}_{k_{n}\in\mathbb{N}}$ are Borel and pairwise disjoint, we see that $\phi$ is
Borel measurable. Thus we have
\begin{align*}
\int_{X}\phi d\mu\leqslant&\nu(Z)\int_{a}^{b}\widehat{h}(t)\rho(t)dt\lesssim\sum_{k_{n}\in\mathbb{N}}\int_{G_{k_{n}}}\frac{\rho(t)}{\int_{G_{k_{n}}}e^{-\epsilon t}dt}dt\\
<&\sum_{k_{n}\in\mathbb{N}}\int_{G_{k_{n}}}\frac{2^{-k_{n}}e^{-\epsilon t}}{\int_{G_{k_{n}}}e^{-\epsilon t}dt}dt\\
= & \sum_{k_{n}\in\mathbb{N}}2^{-k_{n}}<\infty
\end{align*}
and
$$
\int_{\gamma}\phi ds=\int_{a}^{b}\widehat{h}(t)e^{-\epsilon t}dt=\sum_{k_{n}\in\mathbb{N}}1=\infty.
$$
We conclude from the discussions as above that for every $1\leq p<\infty$, there is a function $\phi$ satisfying our requirement. This proves the claim, and hence, the proof of the implication $(2)\Rightarrow(1)$ is complete.
\end{proof}

Let $u\in L_{\text{loc}}^{1}(X, \mu)$.
A Borel function $g: X\rightarrow[0, \infty]$ is an {\it upper gradient} of $u$ if for each nonconstant compact rectifiable curve $\gamma$ in $X$, we have
\begin{equation}\label{gu-up}
|u(x)-u(y)|\leq\int_{\gamma}gds,
\end{equation}
where $x$ and $y$ denote the endpoints of $\gamma$.  The above inequality should be interpreted as also requiring that $\int_{\gamma}gds=\infty$ if either $u(x)$ or $u(y)$ is not finite. Obviously, $g=\infty$ is an upper gradient for each $u\in L_{\text{loc}}^{1}(X, \mu)$.

If \eqref{gu-up} holds for $p$-almost every curve, then $g$ is called a {\it $p$-weak upper gradient} of $u$.
Here, we say that a property holds for $p$-almost every curve $\gamma$ in $X$ if the family $\Gamma$ of all nonconstant compact rectifiable curves for which the property fails has zero $p$-modulus. The $p$-weak upper gradients were introduced by Koskela and MacManus in \cite{KMA}.

If $u$ has an upper gradient in $L^{p}(X, \mu)$, then it has a {\it minimal $p$-weak upper gradient}, denoted by $g_u$, in the sense that $g_{u}\leq g$ a.e. (i.e., almost everywhere) for every $p$-weak upper gradient $g\in L^{p}(X, \mu)$ of $u$. Moreover, $g_u$ is unique up to sets of measure zero. See \cite[Theorem 2.5]{BB} or \cite[Theorem 7.16]{HAJ} for these discussions.
We also refer interested readers to \cite{HK, BBS, NS, NS00} for more discussions on upper gradients and $p$-weak upper gradients.

\begin{defn}\label{def-2.5}
For $1\leq p<\infty$, we define the {\it Sobolev space} $N^{1,p}(X)$ as the normed space of all $u\in L^{p}(X, \mu)$ such that
$$
\|u\|_{N^{1,p}(X)}=\|u\|_{L^{p}(X, \mu)}+\inf_{g}\|g\|_{L^{p}(X, \mu)}<\infty,
$$
where the infimum is taken over all upper gradients $g$ of $u$, or equivalently all $p$-weak upper gradients $g$ of $u$.
\end{defn}

By Proposition \ref{Prop-2.3}, together with \cite[Definition 7.2 and Lemma 7.6]{HAJ}, we see that
any function $u\in L_{\text{loc}}^{1}(X, \mu)$ with an upper gradient $0\leq g\in L^{p}(X, \mu)$ is locally absolutely continuous.
 This fact implies that all such functions $u$ are absolutely continuous on each edge in hyperbolic fillings.

\begin{defn}\label{def-2.6}
For $1\leq p<\infty$, we define the {\it homogeneous Sobolev space} $\dot N^{1,p}(X)$ as the collection of all continuous functions $u$ which satisfy that each $u$ has an upper gradient $0\leq g\in L^{p}(X, \mu)$ and its homogeneous $\dot N^{1,p}$-norm $\|u\|_{\dot N^{1,p}(X)}$ is finite, where
$$
\|u\|_{\dot N^{1,p}(X)}=|u(v_{0})|+\inf_{g}\|g\|_{L^{p}(X, \mu)},
$$
and the infimum is taken over all upper gradients $g$ of $u$, or equivalently all $p$-weak upper gradients $g$ of $u$.
\end{defn}

\begin{defn}\label{def-2.7}
For $1\leq p<\infty$, the $p$-capacity of a set $E\subset X$ is the number
$$
C_{p}(E)=\inf_{u}\|u\|_{N^{1,p}(X)}^{p},
$$
where the infimum is taken over all $u\in N^{1,p}(X)$ such that $u\geq1$ on $E$.
\end{defn}

\begin{rem}\label{rem-2.1}
From Proposition \ref{Prop-2.3}, it is easy to see that if the function $e^{-\epsilon t}\rho(t)^{-1/p}$ belongs to $L^{p^{\prime}}_{\text{loc}}([0,\infty))$, then $\text{Mod}_{p}\Gamma_{x}>0$ for any point $x\in X$. Here the rectifiable curve family $\Gamma_{x}=\{\gamma: x\in\gamma\}$. Hence, by \cite[Proposition 1.48]{BB}, each point in $X$ has a positive $p$-capacity.
\end{rem}

The following is an analog of \cite[Lemma 3.1]{KNW} in the case of hyperbolic fillings.

\begin{lem}\label{lem-2.8}
Let $1\leq p<\infty$, and let $\rho:[0,\infty)\rightarrow(0,\infty)$ be a locally integrable Borel function. Suppose that $\varphi$ is a nonnegative function defined on $[0,\infty)$. Then
\begin{equation*}
\int_{Z}\int_{[v_{0},\xi)}\frac{\varphi(|x|)^{p}}{\nu(B_{v_x})}d\mu(x)d\nu(\xi)\approx\int_{X}\varphi(|x|)^{p}d\mu(x),
\end{equation*}
and for any $g\in L^{p}(X, \mu)$,
\begin{equation}\label{2.8-2}
\int_{Z}\int_{[v_{0},\xi)}\frac{|g(x)|^{p}}{\nu(B_{v_x})}d\mu(x)d\nu(\xi)\lesssim\int_{X}|g(x)|^{p}d\mu(x).
\end{equation}
Here, we recall that $[v_{0},\xi)$ denotes a geodesic ray from $v_{0}$ to $\xi$.
\end{lem}
\begin{proof}
 Let $X=X_{V}\cup X_{H}$, where $X_{V}$ and $X_{H}$ denote the collections of the vertical edges and the horizontal edges on $X$, respectively.

Let $[v_{0},\xi)$ be a geodesic ray from $v_{0}$ to $\xi$, and let $v_{n}(\xi)$
be the vertex such that $v_{n}(\xi)\in[v_{0},\xi)$ and $v_{n}(\xi)\in V_{n}(\xi)$ for all $n\geq0$. By Proposition \ref{prop-2.1}, for all $\xi\in Z$ and $n\geq0$,
\begin{equation}\label{lem-28-1}
1\leq\#V_{n}(\xi)\leq\#\{w:w\sim v_n(\xi)\} \leq K.
\end{equation}
Since $\nu$ is a doubling measure, if $w\sim v$, then $\nu(B_{v})\approx\nu(B_{w})$.
It follows from \eqref{lem-28-1} that
\begin{align*}
\int_{[v_{n}(\xi),v_{n+1}(\xi)]}\frac{\varphi(|x|)^{p}}{\nu(B_{v_x})}d\mu(x)
&\leqslant\sum_{v\in V_{n}(\xi)}\sum_{V_{n+1}(\xi)\ni w\sim v}\int_{[v,w]}\frac{\varphi(|x|)^{p}}{\nu(B_{v_x})}d\mu(x)\\
&\leqslant C(C_\nu, \alpha, \tau, K)\int_{[v_{n}(\xi),v_{n+1}(\xi)]}\frac{\varphi(|x|)^{p}}{\nu(B_{v_x})}d\mu(x).
\end{align*}
For vertices $v\in V_{n}(\xi)$ and $w\in V_{n+1}(\xi)$, we have $\xi\in B_{v}\cap B_{w}$ and $\chi_{[v,w)}(x)\neq0$ only if $x\in[v,w)\subseteq X_{V}$. It follows from the Fubini-Tonelli theorem that
\begin{align}\label{lem-28-2}
\int_{Z}\int_{[v_{0},\xi)}\frac{\varphi(|x|)^{p}}{\nu(B_{v_x})}d\mu(x)d\nu(\xi)
\approx&\int_{Z}\sum_{n=0}^{\infty}\int_{[v_{n}(\xi),v_{n+1}(\xi)]}\frac{\varphi(|x|)^{p}}{\nu(B_{v_x})}d\mu(x)d\nu(\xi)\notag\\
\approx&\int_{Z}\sum_{n=0}^{\infty}\sum_{v\in V_{n}(\xi)}\sum_{V_{n+1}(\xi)\ni w\sim v}\int_{[v,w]}\frac{\varphi(|x|)^{p}}{\nu(B_{v_x})}d\mu(x)d\nu(\xi)\notag\\
=&\int_{X_{V}}\varphi(|x|)^{p}\int_{Z}\sum_{n=0}^{\infty}\sum_{v\in V_{n}(\xi)}\sum_{V_{n+1}(\xi)\ni w\sim v}\frac{\chi_{[v,w]}(x)}{\nu(B_{v_x})}d\nu(\xi)d\mu(x)\notag\\
=:&\int_{X_{V}}\varphi(|x|)^{p}Q(x)d\mu(x),
\end{align}
where
$$
Q(x)=\int_{Z}\sum_{n=0}^{\infty}\sum_{v\in V_{n}(\xi)}\sum_{V_{n+1}(\xi)\ni w\sim v}\frac{\chi_{[v,w]}(x)}{\nu(B_{v_x})}d\nu(\xi).
$$

Now, we estimate $Q(x)$. Recall that $A_{n}\subseteq A_{m}$ when $m\geq n\geq0$. For any $v=(z,n)\in V_{n}$, let $v_{d}:=(z,n+1)\in V_{n+1}$. Then $B_{v_{d}}\subseteq B_{v}$ and $v_{d}\sim v$. If there is some $v\in V$ such that $x\in [v, v_{d})$, then
\begin{equation}\label{lem-28-3}
Q(x)\gtrsim \frac{\nu(B_{v_{d}})}{\nu(B_{v})}\geqslant C(C_\nu, \alpha, \tau).
\end{equation}
It follows from \eqref{lem-28-2}, \eqref{lem-28-3} and the fact $\cup_{v\in V}[v, v_{d})\subseteq X_{V}$ that
\begin{align*}
\int_{Z}\int_{[v_{0},\xi)}\frac{\varphi(|x|)^{p}}{\nu(B_{v_x})}d\mu(x)d\nu(\xi)
\gtrsim\sum_{v\in V}\int_{[v, v_{d})}\varphi(|x|)^{p}Q(x)d\mu(x)
\gtrsim\sum_{v\in V}\int_{[v, v_{d})}\varphi(|x|)^{p}d\mu(x).
\end{align*}
By Proposition \ref{prop-2.1}, we have
\begin{equation*}
\sum_{v\in V}\int_{[v, v_{d})}\varphi(|x|)^{p}d\mu(x)
\leqslant\int_{X_{V}}\varphi(|x|)^{p}d\mu(x)
\leqslant C(C_\nu, \alpha, \tau, K)\sum_{v\in V}\int_{[v, v_{d})}\varphi(|x|)^{p}d\mu(x).
\end{equation*}
Then we get
\begin{equation}\label{lem-28-4}
\int_{Z}\int_{[v_{0},\xi)}\frac{\varphi(|x|)^{p}}{\nu(B_{v_x})}d\mu(x)d\nu(\xi)\gtrsim\int_{X_{V}}\varphi(|x|)^{p}d\mu(x).
\end{equation}

For some $v\in V_{n}$ and $V_{n+1}\ni w\sim v$, if $x\in[v,w]$, then the doubling property of $\nu$ implies that
\begin{equation}\label{lem-28-5}
Q(x)\lesssim \frac{\nu(\cup_{V_{n+1}\ni w\sim v}B_{w})}{\nu(B_{v})}\leqslant C(C_\nu, \alpha, \tau).
\end{equation}
It follows from \eqref{lem-28-2}, \eqref{lem-28-5} and the fact $X_{V}=\bigcup_{n\geqslant0}\bigcup_{v\in V_{n}}\bigcup_{V_{n+1}\ni w\sim v}[v, w]$ that
\begin{align}\label{lem-28-6}
\int_{Z}\int_{[v_{0},\xi)}\frac{\varphi(|x|)^{p}}{\nu(B_{v_x})}d\mu(x)d\nu(\xi)
\lesssim\int_{X_{V}}\varphi(|x|)^{p}Q(x)d\mu(x)
\lesssim\int_{X_{V}}\varphi(|x|)^{p}d\mu(x).
\end{align}

Since $\varphi$ is a nonnegative function related to the graph distance $|x|=|x-v_{0}|$,
it follows from Proposition \ref{prop-2.1} that
\begin{equation}\label{lem-28-7}
\int_{X}\varphi(|x|)^{p}d\mu(x)\approx\int_{X_{V}}\varphi(|x|)^{p}d\mu(x).
\end{equation}

By combining \eqref{lem-28-4}, \eqref{lem-28-6} and \eqref{lem-28-7}, we get
$$
\int_{Z}\int_{[v_{0},\xi)}\frac{\varphi(|x|)^{p}}{\nu(B_{v_x})}d\mu(x)d\nu(\xi)\approx\int_{X}\varphi(|x|)^{p}d\mu(x).
$$

For the proof of the second assertion, assume that $g\in L^{p}(X, \mu)$. Then it follows from the Fubini theorem and \eqref{lem-28-5} that
\begin{align*}
\int_{Z}\int_{[v_{0},\xi)}\frac{|g(x)|^{p}}{\nu(B_{v_x})}d\mu(x)d\nu(\xi)
\approx&\int_{Z}\sum_{n=0}^{\infty}\int_{[v_{n}(\xi),v_{n+1}(\xi)]}\frac{|g(x)|^{p}}{\nu(B_{v_x})}d\mu(x)d\nu(\xi)\\
=&\int_{X_{V}}|g(x)|^{p}\int_{Z}\sum_{n=0}^{\infty}\frac{\chi_{[v_{n}(\xi),v_{n+1}(\xi)]}(x)}{\nu(B_{v_x})}d\nu(\xi)d\mu(x)\\
\lesssim&\int_{X_{V}}|g(x)|^{p}d\mu(x)
\leqslant\int_{X}|g(x)|^{p}d\mu(x).
\end{align*} Hence the lemma is proved.
\end{proof}

\begin{rem}\label{rem-2.2}

By choosing $\varphi(t)=1$ on $[0,\infty)$ in Lemma \ref{lem-2.8}, we can get
$$
\mu(X)=\int_{X}d\mu\approx\int_{X_{V}}d\mu\approx\nu(Z)\int_{0}^{\infty}\rho(t)dt.
$$
This implies that
 $\int_{0}^{\infty}\rho(t)dt=\infty$ if and only if $\mu(X)=\infty$.
\end{rem}

\section{Proof of Theorem \ref{thm-1.1}}\label{sec-3}
In this section, let $ds=e^{-\epsilon|x|}d|x|$ with $\epsilon=\log\alpha>0$. Also, we assume that $1\leq p<\infty$ and $\rho\in\mathcal F_p$.  Further, by \cite[Proposition 4.4]{BBS}, we can replace $\partial X$, the boundary of $X$, by $Z$. For $\xi\in Z$, let $v_{0}=v_{0}(\xi)$ and let $[v_0, \xi)$ denote a geodesic ray from $v_0$ to $\xi$ with
\beq\label{lem-31-30}
[v_{0},\xi):=\bigcup_{n\geq0}[v_{n}(\xi),v_{n+1}(\xi)],
\eeq
where $v_{n}(\xi)\in V_{n}(\xi)$ and $v_{n+1}(\xi)\in V_{n+1}(\xi)$.

 We develop our arguments based on three cases: $R_{p,\rho}<\infty$, $R_{p,\rho}=\infty$ and $\mu(X)=\infty$.

\begin{lem}\label{lem-3.1}
 Suppose that $R_{p,\rho}<\infty$. Then the trace operator $\mathscr T$ defined in \eqref{equ-wetrf} exists, and
it is a bounded operator
from $\dot N^{1,p}(X)$ to $L^{p}(Z,\nu)$, and also, a bounded operator
from $N^{1,p}(X)$ to $L^{p}(Z,\nu)$.
\end{lem}

\begin{proof} Assume that $R_{p,\rho}<\infty$. To prove the first part of the lemma, let $u\in \dot N^{1,p}(X)$. We show that the trace $\mathscr T u$ exist $\nu$-a.e. in $Z$, and that $$\|\mathscr T u\|_{L^{p}(Z,\nu)}\lesssim\|u\|_{\dot N^{1,p}(X)}.$$

Let
\begin{equation*}
u^{\ast}(\xi)=|u(v_{0})|+\int_{[v_{0}, \xi)}g_{u}ds
\end{equation*} in $Z$,
where $g_{u}$ is a minimal $p$-weak upper gradient of $u$.

First, we show that $u^{\ast}\in L^{p}(Z,\nu)$.
For each $n$, we deduce from \eqref{ds-du} that
$$
\int_{[v_{n}(\xi),v_{n+1}(\xi)]}g_{u}ds\approx\int_{[v_{n}(\xi),v_{n+1}(\xi)]}g_{u}(x)\frac{e^{-\epsilon|x|}}{\rho(|x|)\nu(B_{v_x})}d\mu(x).
$$

We separate the arguments into the following two cases: $p>1$ and $p=1$. For the former, it follows from the H\"{o}lder inequality that
\begin{align*}
u^{\ast}(\xi)^{p}\lesssim&|u(v_{0})|^{p}+\bigg(\int_{[v_{0}, \xi)}g_{u}ds\bigg)^{p}\\
\lesssim&|u(v_{0})|^{p}+\bigg(\int_{[v_{0},\xi)}\frac{g_{u}(x)^{p}}{\nu(B_{v_x})}d\mu(x)\bigg)\bigg(\int_{[v_{0},\xi)}e^{-\frac{\epsilon p|x|}{p-1}}\rho(|x|)^{-\frac{p}{p-1}}\nu(B_{v_x})^{-1}d\mu(x)\bigg)^{p-1}\\
\lesssim&|u(v_{0})|^{p}+\bigg(\int_{0}^{\infty}e^{-\frac{\epsilon p}{p-1}t}\rho(t)^{\frac{1}{1-p}}dt\bigg)^{p-1}\int_{[v_{0},\xi)}\frac{g_{u}(x)^{p}}{\nu(B_{v_x})}d\mu(x).
\end{align*}
Thus, by Lemma \ref{lem-2.8} and the assumption $R_{p,\rho}<\infty$, we can get
\begin{align*}
\int_{Z}|u^{\ast}(\xi)|^{p}d\nu(\xi)
\lesssim&|u(v_{0})|^{p}\nu(Z)+R_{p,\rho}^{p-1}\int_{Z}\int_{[v_{0},\xi)}\frac{g_{u}(x)^{p}}{\nu(B_{v_x})}d\mu(x)d\nu(\xi)\\
\lesssim&|u(v_{0})|^{p}+\int_{X}g_{u}(x)^{p}d\mu(x).
\end{align*}

For the latter, that is, $p=1$, again, by Lemma \ref{lem-2.8} and the assumption $R_{1,\rho}<\infty$, we obtain
\begin{align*}
\int_{Z}|u^{\ast}(\xi)|d\nu(\xi)
\lesssim&|u(v_{0})|+\int_{Z}\int_{[v_{0},\xi)}g_{u}(x)\frac{e^{-\epsilon|x|}}{\rho(|x|)\nu(B_{v_x})}d\mu(x)d\nu(\xi)\\
\lesssim&|u(v_{0})|+\|e^{-\epsilon t}\rho(t)^{-1}\|_{L^{\infty}([0, \infty))}\int_{Z}\int_{[v_{0},\xi)}\frac{g_{u}(x)}{\nu(B_{v_x})}d\mu(x)d\nu(\xi)\\
\lesssim&|u(v_{0})|+\int_{X}g_{u}d\mu.
\end{align*}
Then for all $1\leq p<\infty$, we have the estimate:
$$
\|u^{\ast}\|_{L^{p}(Z,\nu)}
\lesssim|u(v_{0})|+\bigg(\int_{X}g_{u}^{p}d\mu\bigg)^{\frac{1}{p}}
=\|u\|_{\dot N^{1,p}(X)}.
$$
Hence $u^{\ast}\in L^{p}(Z,\nu)$ for $1\leq p<\infty$. Also, we know that for $\nu$-almost every $\xi\in Z$, $u^{\ast}(\xi)<\infty$.

Second, we demonstrate that if $\xi\in Z$ satisfies $u^{\ast}(\xi)<\infty$, then $\mathscr T u(\xi)$ exists and is independent of the choice of the geodesic rays $[v_{0},\xi)$.

Let $[v_{0},\xi)$ denote a geodesic ray from $v_{0}$ to $\xi$ with the expression \eqref{lem-31-30}. Since $u^{\ast}(\xi)<\infty$,
we have
$$
\int_{[v_{n}(\xi),\xi)}g_{u}ds\rightarrow0 \ \ \text{as} \ \ n\rightarrow\infty.
$$
For any $\varepsilon>0$, there exists an integer $N>0$ such that for any $x_{1}, x_{2}\in [v_{0}, \xi)$, if $|x_{1}|>N$ and $|x_{2}|>N$, then
$$
|u(x_{1})-u(x_{2})|\leqslant\int_{[x_{1},x_{2}]}g_{u}ds<\varepsilon,
$$
where $[x_{1},x_{2}]\subset [v_{0}, \xi)$. Then Cauchy's convergence test implies that the limit $\mathscr T u(\xi)$ exists along $[v_{0},\xi)$.

To prove that $\mathscr T u(\xi)$ is independent of the choice of $[v_{0},\xi)$, let $J_{1}(\xi)$ and $J_{2}(\xi)$ be two geodesic rays from $v_0$ and $\xi$ such that
$$
\lim_{J_{1}(\xi)\ni x\rightarrow\xi}u(x)=A_{1} \ \ \ \text{and} \ \ \ \lim_{J_{2}(\xi)\ni x\rightarrow\xi}u(x)=A_{2}
$$ with $A_{1}\not=A_{2},$ where
$$
J_{1}(\xi)=\bigcup_{n\geq 0}[v_{n}^{1}(\xi),v_{n+1}^{1}(\xi)] ,\;\; J_{2}(\xi)=\bigcup_{n\geq0}[v_{n}^{2}(\xi),v_{n+1}^{2}(\xi)],
$$
and $v_{n}^{i}(\xi)\in V_{n}(\xi)$ for $i=1,2$.
Note that for any $v\in V_{n}(\xi)$ and $w\in V_{n+1}(\xi)$, $w\sim v$. Take
$$
J=\bigcup_{n\geq 0}([v_{2n}^{1}(\xi),v_{2n+1}^{2}(\xi)]\cup[v_{2n+1}^{2}(\xi),v_{2n+2}^{1}(\xi)]).
$$
Then the limit $\mathscr T u(\xi)$ does not exist along the geodesic ray $J$. This contradiction shows that the assertion is true.
Now, we conclude that the trace function $\mathscr T u$ exists for $\nu$-a.e. in $Z$.

Third, we show that $\mathscr T u\in L^{p}(Z, \nu)$ for $u\in \dot N^{1,p}(X)$.
For this, let $g_{u}$ be a minimal $p$-weak upper gradient of $u$.
Let $\xi\in Z$ and $x\in[v_{0}, \xi)$, and let $[v_{0}, x]\subset [v_{0}, \xi)$ be the geodesic segment joining $v_0$ to $x$. By the definition of minimal $p$-weak upper gradients, we have
$$
|u(x)|-|u(v_{0})|\leq |u(x)-u(v_{0})|\leq \int_{[v_{0}, x]}g_{u}ds\leq \int_{[v_{0}, \xi)}g_{u}ds.
$$
If $\mathscr T u(\xi)$ exists, then $|u(x)|\rightarrow|\mathscr T u(\xi)|$ as $x\rightarrow\xi$.
 Hence, for $\nu$-almost every $\xi\in Z$,
$$
|\mathscr T u(\xi)|\leqslant|u(v_{0})|+\int_{[v_{0}, \xi)}g_{u}ds=u^{\ast}(\xi),
$$
and thus,
\beq\label{lem-31-1}
\|\mathscr T u\|_{L^{p}(Z,\nu)}\lesssim|u(v_{0})|+\bigg(\int_{X}g_{u}^{p}d\mu\bigg)^{\frac{1}{p}}=\|u\|_{\dot N^{1,p}(X)}<\infty,
\eeq since $\bigg(\int_{X}g_{u}^{p}d\mu\bigg)^{\frac{1}{p}}=\inf_{g}\|g\|_{L^{p}(X, \mu)}$.

These show that $\mathscr T u\in L^{p}(Z, \nu)$ and $\mathscr T: \dot N^{1,p}(X)\rightarrow L^{p}(Z,\nu)$ is a bounded operator. Hence the lemma is proved.

Next, we are going to show the second part of the lemma, that is, $\mathscr T: N^{1,p}(X)\rightarrow L^{p}(Z,\nu)$ is bounded.
Observe that  the continuity of $u$ is not required in the proof of the inequality \eqref{lem-31-1}. Then it follows that for any $u\in N^{1,p}(X)$,
\beqq
\|\mathscr T u\|_{L^{p}(Z,\nu)}\lesssim|u(v_{0})|+\bigg(\int_{X}g_{u}^{p}d\mu\bigg)^{\frac{1}{p}},
\eeqq
and so, to prove this claim, we only need to show that for all $1\leq p<\infty$ and all $u\in N^{1,p}(X)$,
\beq\label{lem-31-3}
|u(v_{0})|\lesssim\|u\|_{N^{1,p}(X)}.
\eeq

To reach this goal, let $u\in N^{1,p}(X)$. If $u(v_{0})=0$, then the assertion is true. Assume that $u(v_{0})\neq 0$. Since the function $u(x)/u(v_{0})$ satisfies $u(x)/u(v_{0})=1$ on the set $\{v_{0}\}$, by Definition \ref{def-2.7}, we have that for all $1\leq p<\infty$,
$$
|u(v_{0})|^{p}C_{p}(\{v_{0}\})\leq\|u\|_{N^{1,p}(X)}^{p}.
$$
From Remark \ref{rem-2.1}, we know that $C_{p}(\{v_{0}\})>0$, and then,
$$
|u(v_{0})|\lesssim\|u\|_{N^{1,p}(X)}.
$$
This proves \eqref{lem-31-3}, which indicates that the second part of the lemma is true, and hence, the lemma is proved.
\end{proof}

\begin{lem}\label{lem-3.2} Suppose that $R_{p,\rho}=\infty$.\ben
\item[$(1)$]
If $1\leq p<\infty$, then there exists a function $u\in \dot N^{1,p}(X)$ such that for any $\xi\in Z,$
$$
\lim_{[v_{0},\xi)\ni x\rightarrow\xi}u(x)=+\infty.
$$
\item[$(2)$]
If $p=1$, then there exists a function $u\in N^{1,1}(X)$ such that for any $\xi\in Z$, the limit
$
\lim_{[v_{0},\xi)\ni x\rightarrow\xi}u(x)
$
does not exist.

\item[$(3)$]
If $1< p<\infty$ and if $\mu(X)<\infty$, then there exists a function $u\in N^{1,p}(X)$ such that for any $\xi\in Z$, the limit
$
\lim_{[v_{0}, \xi)\ni x\rightarrow\xi}u(x)
$
does not exist.
\een
\end{lem}

\begin{proof}
$(1)$. Since $R_{p,\rho}=\infty$ for $1\leq p<\infty$,
it follows from \cite[Theorem 3.5]{KNW} that there is a nonnegative measurable function $q_1:$ $[0,\infty)\rightarrow[0,\infty]$ such that
$$\int_{0}^{\infty}q_1(t)e^{^{-\epsilon t}}dt=+\infty \;\;\mbox{and}\;\;\int_{0}^{\infty}q_1(t)^{p}\rho(t)dt<+\infty. $$

Let $$u_1(x)=\int_{0}^{|x|}q_1(t)e^{^{-\epsilon t}}dt$$ in $X$, and let $$g_1(x) = q_1(|x|).$$ Then $g_1$ is an upper gradient of $u_1$. It follows from $0<\nu(Z)<\infty$ and Lemma \ref{lem-2.8} that
$$
\|g_1\|_{L^{p}(X, \mu)}^{p}\approx\int_{Z}\int_{[v_{0},\xi)}\frac{q_1(|x|)^{p}}{\nu(B_{v_x})}d\mu(x)d\nu(\xi)\approx\int_{0}^{\infty}q_1(t)^{p}\rho(t)dt<\infty.
$$
These show that the function $u_1$ satisfies the requirement in the first statement of the lemma.

$(2)$. Assume that $R_{1,\rho}=\infty$. For each $k\in\mathbb{N}$, let
$$
E_{k}:=\{t\in[0,\infty):e^{-\epsilon t}\rho(t)^{-1}\geq2^{k}\}.
$$ Then the assumption $R_{1,\rho}=\infty$ implies that for each $k\in\mathbb{N}$,
$m(E_{k})>0$. Otherwise, there exists $k_0\in\mathbb{N}$ such that $e^{-\epsilon t}\rho(t)^{-1}<2^{k_0}$ a.e. in $[0,\infty)$. It is impossible.

Since $e^{-\epsilon t}\rho(t)^{-1}\in L^{\infty}_{\text{loc}}([0,\infty))$, again, the assumption $R_{1,\rho}=\infty$ implies that there exists an increasing sequence $\{t_{k}:t_{k}\in[0,\infty)\}_{k\in\mathbb{N}}$ such that
$
\bigcup_{k\in\mathbb{N}}[t_{k},t_{k+1}]=[0,\infty)
$
and for any $k\in\mathbb{N}$,
\begin{equation}\label{lem-32-1}
m(E_{k}\cap[t_{k},t_{k+1}])>0.
\end{equation}

Let
$$
L_{k}=\int_{t_{k}}^{t_{k+1}}\rho(t)dt.
$$
As $\rho\in L_{\text{loc}}^{1}([0,\infty))$, we see that $0<L_{k}<\infty$.
It follows that $[t_{k},t_{k+1}]$ can be divided into $\lceil2^{k}L_{k}\rceil$ subintervals $\{I_{k, j}\}$ such that their interiors are pairwise disjoint,
$$
\bigcup_{j=1}^{\lceil2^{k}L_{k}\rceil}I_{k, j}=[t_{k},t_{k+1}] \ \  \text{and} \ \  0<\int_{I_{k, j}}\rho(t)dt\leq2^{-k},
$$
where $\lceil 2^{k}L_{k}\rceil$ denotes the least integer greater than $2^{k}L_{k}$. From \eqref{lem-32-1}, we can choose one subinterval $I_{k}\in\{I_{k, j}\}$ satisfying $$m(E_{k}\cap I_{k})>0.$$

We define the function $q_2: [0,\infty)\rightarrow[0,\infty]$ by setting
\begin{equation}\label{lem-32-2}
q_2(t)=\sum_{k\in\mathbb{N}}\frac{2}{\int_{E_{k}\cap I_{k}}e^{-\epsilon t}dt}\chi_{E_{k}\cap I_{k}}(t).
\end{equation}
For any $k\in\mathbb{N}$,
\begin{equation}\label{lem-32-3}
\int_{t_{k}}^{t_{k+1}}q_2(t)e^{-\epsilon t}dt=\int_{E_{k}\cap I_{k}}\frac{2e^{-\epsilon t}}{\int_{E_{k}\cap I_{k}}e^{-\epsilon t}dt}dt=2,
\end{equation}
which implies that there is $t_k^\prime\in(t_{k}, t_{k+1})$ such that
$$
\int_{t_{k}}^{t_k^\prime}q_2(t)e^{-\epsilon t}dt=\int_{t_k^\prime}^{t_{k+1}}q_2(t)e^{-\epsilon t}dt=1.
$$
We define the function $u_2$ as follows: For all $k\in\mathbb{N}$,
\begin{equation*}
u_2(x)=
\left\{\begin{array}{cl}
\int_{t_{k}}^{|x|}q_2(t)e^{-\epsilon t}dt, \,&\text{if} \ \ |x|\in[t_{k}, t_{k}^{\prime})\cap I_{k}, \\
\int_{|x|}^{t_{k+1}}q_2(t)e^{-\epsilon t}dt, \,&\text{if} \ \ |x|\in[t_{k}^{\prime}, t_{k+1})\cap I_{k}, \\
0, \,&\text{if} \ \ |x|\in[t_{k}, t_{k+1})\setminus I_{k}
\end{array}\right.
\end{equation*} in $X$.

Obviously, $0\leq u_2(x)\leq 1$ in $X$, and for each $k$, there must be a pair of points $y_k$ and $z_k$ in $[v_0, \xi)$ such that $$|y_k|=t_k\;\;\mbox{and}\;\;|z_k|=t_k^{\prime}.$$ In this way, we find two sequences $\{y_k\}$ and $\{z_k\}$ of $[v_0, \xi)$ such that\
$$u_2(y_{k})=0\;\;\mbox{and}\;\; u_2(z_{k})=1,$$ and
$$y_k\rightarrow\xi\;\;\mbox{and} \;\;z_k\rightarrow\xi$$ as $k\rightarrow\infty$.

These show that the function $u_2$ has no limit along $[v_0, \xi)$, i.e., the limit $\lim_{[v_{0},\xi)\ni x\rightarrow\xi}u_2(x)$ does not exist.

It remains to prove that $u_2\in N^{1,1}(X)$.  For this, let $$g_2(x):=q_2(|x|),$$ where $q_2$ is defined in \eqref{lem-32-2}. For any $k\in\mathbb{N}$, if $|x|\in [t_{k}, t_{k+1})\setminus I_{k}$, then $g_2(x)=0$. This fact implies that $g_2$ is an upper gradient of $u_2$.
We first estimate the $L^{1}$-norm of the upper gradient $g_2$. It follows from \eqref{lem-32-3}, together with the definitions of the function $g_2$ and the set $E_{k}$, that
\begin{align*}
\int_{X}g_2d\mu\lesssim&\int_{Z}\int_{[v_{0},\xi)}\frac{q_2(|x|)}{\nu(B_{v_x})}d\mu(x)d\nu(\xi)\lesssim \nu(Z)\int_{0}^{\infty}q_2(t)\rho(t)dt \\
\lesssim &\sum_{k\geq0}\int_{E_{k}\cap I_{k}}q_2(t)\rho(t)dt
\lesssim  \sum_{k\geq0}2^{-k}\int_{E_{k}\cap I_{k}}q_2(t)e^{-\epsilon t}dt\\
=&\sum_{k\geq0}2^{1-k}<\infty.
\end{align*}

Note that for any $x\in X$, if $u_2(x)\neq0$, then there is $k\in\mathbb N$ such that $|x|\in I_{k}$. It follows from the fact $0\leq u_2(x)\leq 1$ and Lemma \ref{lem-2.8} that
\begin{align*}
\int_{X}|u_2|d\mu=&\sum_{k\geq0}\int_{\{x\in X:|x|\in I_{k}\}}u_2(x)d\mu(x)\\
\lesssim&\sum_{k\geq0}\int_{Z}\int_{\{x\in [v_{0},\xi):|x|\in I_{k}\}}\frac{u_2(x)}{\nu(B_{v_x})}d\mu(x)d\nu(\xi)\\
\lesssim&\sum_{k\geq0}\nu(Z)\int_{I_{k}}\rho(t)dt\\
\lesssim&\sum_{k\geq0}2^{-k}<\infty,
\end{align*}
from which we conclude that $u_2\in N^{1,1}(X)$.

$(3)$. For $1<p<\infty$, it follows from the assumption $R_{p,\rho}=\infty$ that there exists a sequence $\{[t_{j},t_{j+1}]\}_{j\in\mathbb{N}}$ as the subintervals of $[0,\infty)$ such that $$\bigcup_{j\in\mathbb{N}}[t_{j},t_{j+1}]=[0,\infty),$$ and for all $j\in\mathbb{N}$,
$$
\int_{t_{j}}^{t_{j+1}}e^{-\frac{\epsilon p}{p-1}t}\rho(t)^{\frac{1}{1-p}}dt\geqslant2^{j}.
$$
Also, we see that there is $t_{j}^{\prime}\in(t_{j},t_{j+1})$ such that
$$
\int_{t_{j}}^{t_{j}^{\prime}}e^{-\frac{\epsilon p}{p-1}t}\rho(t)^{\frac{1}{1-p}}dt=\frac{1}{2}\int_{t_{j}}^{t_{j+1}}e^{-\frac{\epsilon p}{p-1}t}\rho(t)^{\frac{1}{1-p}}dt.
$$
Let us define a function $q_3: [0,\infty)\rightarrow[0,\infty]$ by setting
\beq\label{lem-32-4}
q_3(t)=\sum_{j\in\mathbb{N}}\frac{2e^{-\frac{\epsilon t}{p-1}}\rho(t)^{\frac{1}{1-p}}}{\int_{t_{j}}^{t_{j+1}}e^{-\frac{\epsilon p}{p-1}t}\rho(t)^{\frac{1}{1-p}}dt}\chi_{[t_{j},t_{j+1})}(t),
\eeq
and then, we define a function $u_3$ as follows: For $j\in\mathbb{N}$,
\begin{equation*}
u_3(x)=
\left\{\begin{array}{cl}
\int_{t_{j}}^{|x|}q_3(t)e^{-\epsilon t}dt, \,&\text{if} \ \ |x|\in[t_{j},t_{j}^{\prime}), \\
\int_{|x|}^{t_{j+1}}q_3(t)e^{-\epsilon t}dt, \,&\text{if} \ \ |x|\in[t_{j}^{\prime},t_{j+1})
\end{array}\right.
\end{equation*}  in $X$.
Then we have that $0\leq u_3(x)\leq 1$ for $x\in X$, and there exist two sequences $\{y_j\}$ and $\{z_j\}\subset [v_0, \xi)$ such that for each $k$,
$$ u_3(y_k)=1 \;\; \text{and} \;\; u_3(z_k)=0,$$ and
$$y_k\rightarrow\xi \;\; \text{and} \;\; z_k\rightarrow\xi$$ as $k\rightarrow\infty.$
These show that the function $u_3$ has no limit along the geodesic ray $[v_0, \xi)$.

In the following, we show that $u_3\in N^{1,p}(X)$. To reach this goal, let $$g_3(x):=q_3(|x|),$$ where $q_3$ is defined in \eqref{lem-32-4}.
Then $g_3$ is an upper gradient of $u_3$.

 First, we estimate $L^{p}$-norm of $u_3$. Since $\mu(X)<\infty$ and $u_3(x)\in[0,1]$ in $X$, we have
$$
\int_{X}|u_3(x)|^{p}d\mu\leqslant\mu(X)<\infty.
$$

Next, we estimate $L^{p}$-norm of the upper gradient $g_3$. By \eqref{lem-32-4} and elementary calculations, we get that
\begin{align*}
\int_{0}^{\infty}q_3(t)^{p}\rho(t)dt
\lesssim&\sum_{j\in\mathbb{N}}\int_{t_{j}}^{t_{j+1}}\frac{2^{p}e^{-\frac{\epsilon p}{p-1}t}\rho(t)^{\frac{1}{1-p}}}{(\int_{t_{j}}^{t_{j+1}}e^{-\frac{\epsilon p}{p-1}t}\rho(t)^{\frac{1}{1-p}}dt)^{p}}dt\\
=&\sum_{j\in\mathbb{N}}2^{p}\bigg(\int_{t_{j}}^{t_{j+1}}e^{-\frac{\epsilon p}{p-1}t}\rho(t)^{\frac{1}{1-p}}dt\bigg)^{1-p}\\
\leqslant&\sum_{j\in\mathbb{N}}2^{p}2^{j(1-p)}<\infty.
\end{align*}
It follows from Lemma \ref{lem-2.8} that
\begin{align*}
\int_{X}g_3(x)^{p}d\mu
\approx\int_{Z}\int_{[v_{0},\xi)}\frac{q_3(|x|)^{p}}{\nu(B_{v_x})}d\mu(x)d\nu(\xi)
\lesssim\nu(Z)\int_{0}^{\infty}q_3(t)^{p}\rho(t)dt<\infty.
\end{align*}
This guarantees that $u_3\in N^{1,p}(X)$.
\end{proof}

\begin{lem}\label{lem-3.3}  Suppose that $\mu(X)=\infty$.
If $1< p<\infty$, then the following conditions are equivalent: \ben
\item[$(a)$]
$\mathcal{R}_{p,\rho}<\infty$;
\item[$(b)$]
$\mathscr Tf$ exists for every $f\in N^{1,p}(X)$ and $\mathscr Tf(\xi)=0$ for $\nu$-almost every $\xi\in Z$;
\item[$(c)$]
$\mathscr Tf$ exists for every $f\in N^{1,p}(X)$.
\een
\end{lem}

\begin{proof}  Assume that $\mu(X)=\infty$ and that $p>1$. Since the implication $(b)\Rightarrow(c)$ is obvious, to prove this part of the lemma, it suffices to show the implications $(a)\Rightarrow(b)$ and $(c)\Rightarrow(a)$. We start with the proof of $(a)\Rightarrow(b)$.

$(a)\Rightarrow(b)$. Assume that $\mathcal{R}_{p,\rho}<\infty$.
Let $f\in N^{1,p}(X)$. It follows from \eqref{2.8-2} in Lemma \ref{lem-2.8} that
\begin{equation}\label{lem-33-1}
\int_{Z}\int_{[v_{0},\xi)}\frac{|f(x)|^{p}}{\nu(B_{v_x})}d\mu(x)d\nu(\xi)\lesssim\int_{X}|f(x)|^{p}d\mu(x)<\infty.
\end{equation}

Let $g_{f}$ denote a minimal $p$-weak upper gradient of $f$. Similarly, we get
\begin{equation}\label{lem-33-2}
\int_{Z}\int_{[v_{0},\xi)}\frac{|g_{f}(x)|^{p}}{\nu(B_{v_x})}d\mu(x)d\nu(\xi)\lesssim\int_{X}|g_{f}(x)|^{p}d\mu(x)<\infty.
\end{equation}
By \eqref{lem-33-1} and \eqref{lem-33-2}, for $\nu$-almost every $\xi\in Z$, we have
$$
\int_{[v_{0},\xi)}\frac{|f(x)|^{p}}{\nu(B_{v_x})}d\mu(x)<\infty
\ \ \ \text{and} \ \ \
\int_{[v_{0},\xi)}\frac{|g_{f}(x)|^{p}}{\nu(B_{v_x})}d\mu(x)<\infty.
$$

Let $\{O_{k}\}_{k=1}^{\infty}$ be a sequence of subintervals of $[0,\infty)$ as in the definition of $\mathcal{R}_{p,\rho}$. Then for each $k$, $$\int_{O_{k}}\rho(t)dt=1,$$ and for $\nu$-almost every $\xi\in Z$,
\begin{equation*}
\sum_{k=1}^{\infty}\int_{O_{k}}|f(x(t))|^{p}\rho(t)dt\approx\int_{[v_{0},\xi)}\frac{|f(x)|^{p}}{\nu(B_{v_x})}d\mu(x)<\infty,
\end{equation*}
where $x(t)$ denotes the point in $[v_{0},\xi)$ with $|x(t)|=t$.

Similarly, we have
\begin{equation*}
\sum_{k=1}^{\infty}\int_{O_{k}}|g_{f}(x(t))|^{p}\rho(t)dt\approx\int_{[v_{0},\xi)}\frac{|g_{f}(x)|^{p}}{\nu(B_{v_x})}d\mu(x)<\infty.
\end{equation*}
These guarantee that
\begin{equation}\label{lem-33-5}
\lim_{k\rightarrow\infty}\int_{O_{k}}|f(x(t))|^{p}\rho(t)dt=0 \ \ \text{and} \ \ \lim_{k\rightarrow\infty}\int_{O_{k}}|g_{f}(x(t))|^{p}\rho(t)dt=0.
\end{equation}

We claim that there exists a sequence $\{t_k: t_{k}\in O_{k}\}$ consisting of distinct elements such that
$$
|f(x(t_{k}))|\rightarrow0
$$
as
$k\rightarrow\infty$.
Otherwise, there is $\delta_0>0$ such that for any $N\geq1$, there exists $k_0>N$ satisfying that for all $t\in O_{k_0}$,
$$
|f(x(t))|\geq\delta_0,
$$
which ensures that
$$
\int_{O_{k_0}}|f(x(t))|^{p}\rho(t)dt\geq \delta_0^{p}\int_{O_{k_0}}\rho(t)dt=\delta_0^{p},
$$
since $\int_{O_{k}}\rho(t)dt=1$ for each $k$. This contradicts \eqref{lem-33-5}, and so, the claim is proved.

Since $\mathcal{R}_{p,\rho}<\infty$, it follows from the H\"{o}lder inequality that
\begin{align*}
\sup_{t\in O_{k}}|f(x(t))|\leqslant&|f(x(t_{k}))|+\int_{O_{k}}g_{f}(x(t))e^{-\epsilon t}dt\\
\leqslant&|f(x(t_{k}))|+\bigg(\int_{O_{k}}g_{f}(x(t))^{p}\rho(t)dt\bigg)^{\frac{1}{p}}\bigg(\int_{O_{k}}e^{-\frac{\epsilon p}{p-1}t}\rho(t)^{\frac{1}{1-p}}dt\bigg)^{\frac{p-1}{p}} \,\,\, (\text{by }\eqref{lem-33-5})\\
\lesssim&|f(x(t_{k}))|+\mathcal{R}_{p,\rho}^{\frac{p-1}{p}}\bigg(\int_{O_{k}}g_{f}(x(t))^{p}\rho(t)dt\bigg)^{\frac{1}{p}}\rightarrow0
\end{align*}
as $k\rightarrow\infty$.

Since it follows from $\bigcup_{k=1}^{\infty}O_{k}=[0,\infty)$ that
$$
\bigg|\limsup_{[v_{0},\xi)\ni x\rightarrow\xi}f(x)\bigg| \leq \lim_{k\rightarrow\infty}\sup_{t\in O_{k}}|f(x(t))|,
$$
we infer that for $\nu$-almost every $\xi\in Z$,

$$
\limsup_{[v_{0},\xi)\ni x\rightarrow\xi}f(x)=0.
$$
Hence the arbitrariness of geodesic rays $[v_{0},\xi)$ implies that $\mathscr Tf$ exists, and $\mathscr Tf(\xi)=0$ for $\nu$-almost every $\xi\in Z$.

$(c)\Rightarrow(a)$. Assume that $\mathscr Tf$ exists for every $f\in N^{1,p}(X)$. We check this implication by contradiction. Suppose that $\mathcal{R}_{p,\rho}=\infty$. Then there is a sequence of subintervals $\{O_{k}\}_{k=1}^{\infty}$ with $\bigcup_{k=1}^{\infty}O_{k}=[0,\infty)$ such that
$$
\int_{O_{k}}\rho(t)dt=1 \ \ \  \text{and} \ \ \  \sup_{k\geq1}\int_{O_{k}}e^{-\frac{\epsilon p}{p-1}t}\rho(t)^{\frac{1}{1-p}}dt=\infty.
$$
We select a subsequence of $\{O_{k}\}$, still denoted $\{O_{k}\}$, such that
$$
\int_{O_{k}}e^{-\frac{\epsilon p}{p-1}t}\rho(t)^{\frac{1}{1-p}}dt>4^{k}.
$$
Now, we divide $O_{k}$ into $2^{k}$ intervals
$\{I_{k,l}\}_{l=1}^{2^{k}}$ with $$\int_{I_{k,l}}\rho(t)dt=2^{-k}$$ such that their interior are disjoint. Then there exists at least one interval in $\{I_{k,l}:l=1,2,\ldots,2^{k}\}$ with endpoints $a_k$ and $b_k$ such that
$$
\int_{a_k}^{b_k}\rho(t)dt=2^{-k} \ \ \  \text{and} \ \ \  \int_{a_k}^{b_k}e^{-\frac{\epsilon p}{p-1}t}\rho(t)^{\frac{1}{1-p}}dt>2^{k}.
$$
We define the function $q:[0, \infty)\rightarrow[0, \infty]$ by setting
$$
q(t)=\sum_{k=1}^{\infty}\frac{2e^{-\frac{\epsilon t}{p-1}}\rho(t)^{\frac{1}{1-p}}}{\int_{a_k}^{b_k}e^{-\frac{\epsilon pt}{p-1}}\rho(t)^{\frac{1}{1-p}}dt}\chi_{(a_k, b_k)}(t).
$$
It follows that
\begin{align*}
\int_{0}^{\infty}q(t)^{p}\rho(t)dt=&\sum_{k=1}^{\infty}\int_{a_k}^{b_k}\bigg(\frac{2e^{-\frac{\epsilon t}{p-1}}\rho(t)^{\frac{1}{1-p}}}{\int_{a_k}^{b_k}e^{-\frac{\epsilon pt}{p-1}}\rho(t)^{\frac{1}{1-p}}dt}\bigg)^{p}\rho(t)dt\\
=&\sum_{k=1}^{\infty}2^p\bigg(\int_{a_k}^{b_k}e^{-\frac{\epsilon pt}{p-1}}\rho(t)^{\frac{1}{1-p}}dt\bigg)^{1-p}\\
<&\sum_{k=1}^{\infty}2^p2^{k(1-p)}<\infty,
\end{align*}
and that
$$
\int_{a_k}^{b_k}q(t)e^{-\epsilon t}dt=\int_{a_k}^{b_k}\frac{2e^{-\frac{\epsilon t}{p-1}}\rho(t)^{\frac{1}{1-p}}}{\int_{a_k}^{b_k}e^{-\frac{\epsilon p}{p-1}t}\rho(t)^{\frac{1}{1-p}}dt}e^{-\epsilon t}dt=2
$$
for all $k=1,2,\ldots$.
The for each $k\geq1$, there is $c_{k}\in(a_{k},b_{k})$ such that
$$
\int_{a_k}^{c_k}q(t)e^{-\epsilon t}dt=\int_{c_k}^{b_k}q(t)e^{-\epsilon t}dt=1.
$$

For each $k\geq1$, let
\begin{equation*}
u(x)=
\left\{\begin{array}{cl}
\int_{a_k}^{|x|}q(t)e^{-\epsilon t}dt, \,&\text{if} \ \ |x|\in(a_k, c_k), \\
\int_{|x|}^{b_k}q(t)e^{-\epsilon t}dt, \,&\text{if} \ \ |x|\in[c_k, b_k), \\
0, \,&\text{if} \ \ |x|\in O_k\setminus(a_k, b_k)
\end{array}\right.
\end{equation*} in $X$.
Then for each $k\geq1$, there are a pair of points $y_k$, $z_k$ in $[v_{0},\xi)$ such that $|y_k|=a_k$ and $|z_k|=c_k$, and thus,
$$u(y_k)=0\;\;\mbox{and}\;\; u(z_k)=1.$$

Moreover, let $$g(x):=q(|x|).$$ Then $g$ is an upper gradient of $u$. It follows from Lemma \ref{lem-2.8} and the fact $0<\nu(Z)<\infty$ that
$$
\int_{X}g(x)^{p}d\mu(x)
\approx\int_{Z}\int_{[v_{0},\xi)}\frac{g(x)^{p}}{\nu(B_{v_x})}d\mu(x)d\nu(\xi)
\approx\nu(Z)\int_{0}^{\infty}q(t)^{p}\rho(t)dt<\infty.
$$

Since $u(x)\in[0,1]$ in $X$ and $u(x)\neq0$ only if $|x|\in (a_k, b_k)$, by Lemma \ref{lem-2.8}, we obtain
$$
\int_{X}|u(x)|^{p}d\mu(x)\lesssim\sum_{k=1}^{\infty}\int_{\{x\in X:|x|\in (a_k, b_k)\}}d\mu(x)\lesssim\sum_{k=1}^{\infty}\int_{a_k}^{b_k}\rho(t)dt<\infty.
$$
This implies that $u\in N^{1,p}(X)$, however, it has no limit along any geodesic ray from $v_{0}$ to $\xi$. This contradicts the assumption that $\mathscr Tf$ exists for every $f\in N^{1,p}(X)$.
\end{proof}

\subsection*{Proof of Theorem \ref{thm-1.1}}
\begin{proof}[Proof of (i)]
When $\mu(x)<\infty$, we have that $R_{p,\rho}=\mathcal{R}_{p,\rho}$.
The necessity follows from Lemma \ref{lem-3.1}, and the sufficiency follows from the conclusions $(2)$ and $(3)$ in Lemma \ref{lem-3.2}.
When $\mu(x)=\infty$, the necessity follows from Lemma \ref{lem-3.3}, the fact $\mathcal{R}_{1,\rho}=R_{1,\rho}$ and Lemma \ref{lem-3.1}, and the sufficiency follows from Lemmas \ref{lem-3.2}$(2)$ and \ref{lem-3.3}.
\end{proof}

\begin{proof}[Proof of (ii)]
It follows directly from Lemmas \ref{lem-3.1} and \ref{lem-3.2}$(1)$.
\end{proof}

\section{Relations between $\mathscr Tu$ and $\widetilde{u}$}\label{sec-4}

The purpose of this section is to prove Theorem \ref{thm-imply} and provide the construction of Example \ref{exm-1.1}. We start the definition of the trace functions introduced in \cite{BBS}.

For a function $u\in N^{1, p}(X)$, its trace (function) $\widetilde{u}$ is defined as follows (see \cite[Section 11]{BBS}):
Let
\begin{equation}\label{equ-tracef}
u_{n}(\xi)=\frac{1}{\#A_{n}(\xi)}\sum_{z\in A_{n}(\xi)}u((z, n)).
\end{equation}
If the limit
$$
\widetilde{u}(\xi)=\lim_{n\rightarrow\infty}u_{n}(\xi)
$$
exists for $\nu$-almost every $\xi\in Z$, then we say that $\widetilde{u}$ exists, which is called the {\it trace $($function$)$ of $u$}.

Here, we recall that
$$
A_{n}(\xi)=A_{n}\cap B_{Z}(\xi, \alpha^{-n}) \ \ \ \ \text{and} \ \ \ V_{n}(\xi)=\{(x,n)\in V_{n}:x\in A_{n}(\xi)\}.
$$

\subsection*{Proof of Theorem \ref{thm-imply}}
\begin{proof}
For any $\xi\in Z$, let
$$
E(\xi):=\cup_{n\geq0}E_{n}(\xi),
$$
where $E_{n}(\xi)$ is the set of all vertical edges with endpoints $v_{n}(\xi)\in V_{n}(\xi)$ and $v_{n+1}(\xi)\in V_{n+1}(\xi)$. By Proposition \ref{prop-2.1}, $\#(V_{n}(\xi))\leq K$ for every $n\geq0$. Then there are $\ell$ geodesic rays
$
J_{1}, J_{2}, \ldots, J_{\ell}
$
from $v_{0}$ to $\xi$ such that
$$
E(\xi)=J_{1}\cup J_{2}\cup\ldots\cup J_{\ell},
$$ where $\ell$ is a positive integer not exceeding $K^{2}$.
Note that these geodesics need not be disjoint.

Let $u$ be a measurable function defined on $X$. By the definition of $\mathscr Tu$, we see that for $\nu$-almost every $\xi\in Z$, the limit $\mathscr Tu(\xi)=\lim_{[v_{0},\xi)\ni x\rightarrow\xi}u(x)$
exists, and is independent of the choice of the geodesic rays $J_{i}$, where $1\leq i\leq \ell$. We denote the limit by $A$, i.e., $\mathscr Tu(\xi)=A$. Then for any $\varepsilon>0$ and each $i\in \{1,\ldots,\ell\}$, there exists an integer $N_{i}>0$ such that whenever $x\in J_{i}$ and $|x|>N_{i}$,
$$
|u(x)-A|<\varepsilon.
$$

Let $N=\max\{N_1, N_2, \ldots, N_{\ell}\}$. Then for any $x\in E(\xi)$ with $|x|>N$,
$$
|u(x)-A|<\varepsilon,
$$
which implies that for any $n>N$,
$$
|u_{n}(\xi)-A|<\varepsilon.
$$
Thus the limit $\lim_{n\rightarrow\infty}u_{n}(\xi)$ exists, and also, we have
$$\mathscr Tu(\xi)=\widetilde{u}(\xi)=A.$$ This proves the theorem.
\end{proof}

\subsection*{Construction of Example \ref{exm-1.1}}

Let $1\leq p<\infty$ and $\rho(t)=e^{-t^{2}-\epsilon pt}$. By Lemma \ref{lem-2.8} and the fact $0<\nu(Z)<\infty$, we have
$$
\mu(X)=\int_{X}d\mu\lesssim\int_{0}^{\infty}e^{-t^{2}-\epsilon pt}dt<\infty.
$$
Also, elementary calculations guarantee that $\rho\in L^{p}([0,\infty))$, and for $1\leq p<\infty$, $\mathcal{R}_{p,\rho}=\infty$.

We divide the arguments into two cases: $p>1$ and $p=1$.
Assume that $p>1$. Let
$$
q_1(t)=\sum_{n\in\mathbb N}\frac{2e^{\frac{t^{2}}{p-1}+\epsilon t}}{\int_{n}^{n+1}e^{\frac{t^{2}}{p-1}}dt}\chi_{[n, n+1)}(t)
$$ in $[0,\infty)$.

Moreover, for each $n\in\mathbb{N}$, there is $m'_{n}\in(n, n+1)$ such that
$$
\int_{n}^{m'_{n}}e^{\frac{t^{2}}{p-1}}dt=\frac{1}{2}\int_{n}^{n+1}e^{\frac{t^{2}}{p-1}}dt,
$$
and let
\begin{equation*}
u_1(x)=
\left\{\begin{array}{cl}
\int_{n}^{|x|}q_1(t)e^{-\epsilon t}dt, \,&\text{if} \ \ |x|\in[n, m'_{n}), \\
\int_{|x|}^{n+1}q_1(t)e^{-\epsilon t}dt, \,&\text{if} \ \ |x|\in[m'_{n}, n+1)
\end{array}\right.
\end{equation*} in $X$. Then $0\leq u_1(x)\leq 1$ in $X$, $u_1(x)=0$ for $x\in X$ with $|x|=n$, and $u_1(x)=1$ for $x\in X$ with $|x|=m'_{n}$. This implies that $\mathscr Tu_1(\xi)$ does not exist for any $\xi\in Z$.

However, for each $n\geq0$, $u_1((z, n))=0$ since $|(z, n)|=n$ for $z\in A_n$. This implies that
$$
u_{1, n}(\xi)=\frac{1}{\#A_{n}(\xi)}\sum_{z\in A_{n}(\xi)}u_1((z, n))=0,
$$
and so, $\widetilde{u_1}(\xi)=0$ for every $\xi\in Z$. This implies the existence of $\widetilde{u_1}$.

In the following, we show that $u_1\in N^{1, p}(X)$. First, we estimate the $L^{p}$-norm of $u_1$.
Since $0\leq u_1(x)\leq1$, we have
$$
\int_{X}|u_1(x)|^{p}d\mu\leqslant\mu(X)<\infty.
$$

Let $g_1(x):=q_1(|x|)$ in $X$. Then $g_1$ is an upper gradient of $u_1$. It follows from Lemma \ref{lem-2.8} and $0<\nu(Z)<\infty$ that
\begin{align*}
\int_{X}g_1(x)^{p}d\mu\lesssim&\int_{Z}\int_{[v_{0},\xi)}\frac{q_1(|x|)^{p}}{\nu(B_{v_x})}d\mu_1(x)d\nu(\xi)\\
\lesssim&\int_{0}^{\infty}q_1(t)^{p}e^{-t^{2}-\epsilon pt}dt
=\sum_{n\in\mathbb{N}}\int_{n}^{n+1}\frac{2^{p}e^{\frac{t^{2}}{p-1}}}{(\int_{n}^{n+1}e^{\frac{t^{2}}{p-1}}dt)^{p}}dt\\
\leqslant&\sum_{n\in\mathbb{N}}2^{p}(e^{\frac{n^{2}}{p-1}})^{1-p}=2^{p}\sum_{n\in\mathbb{N}}e^{-n^{2}}<\infty.
\end{align*}
These ensure that $u_1\in N^{1,p}(X)$.

Assume that $p=1$. Let
$$
q_2(t)=\sum_{n\in\mathbb{N}}\frac{2e^{t^{2}+\epsilon t}}{\int_{n}^{n+1}e^{t^{2}}dt}\chi_{[n,n+1)}(t)
$$
in $[0,\infty)$.
Moreover, for each $n\in\mathbb{N}$, there is $m''_{n}\in(n, n+1)$ such that
$$
\int_{n}^{m''_{n}}e^{t^{2}}dt=\frac{1}{2}\int_{n}^{n+1}e^{t^{2}}dt.
$$
Let
\begin{equation*}
u_2(x)=
\left\{\begin{array}{cl}
\int_{n}^{|x|}q_2(t)e^{-\epsilon t}dt, \,&\text{if} \ \ |x|\in[n, m''_{n}), \\
\int_{|x|}^{n+1}q_2(t)e^{-\epsilon t}dt, \,&\text{if} \ \ |x|\in[m''_{n}, n+1)
\end{array}\right.
\end{equation*}  in $X$.
Then $0\leq u_2(x)\leq 1$ in $X$, $u_2(x)=0$ for $x\in X$ with $|x|=n$, and $u_2(x)=1$ for $x\in X$ with $|x|=m''_{n}$. This implies that $\mathscr Tu_2(\xi)$ does not exist for any $\xi\in Z$, and so, $\mathscr Tu_2(\xi)$ does not exist.

However, for each $n\geq0$, $u_2((z, n))=0$ since $|(z, n)|=n$ for $z\in A_n$. This implies that
$$
u_{2, n}(\xi)=\frac{1}{\#A_{n}(\xi)}\sum_{z\in A_{n}(\xi)}u_2((z, n))=0,
$$
and so, $\widetilde{u_2}(\xi)=0$ for every $\xi\in Z$. This implies the existence of $\widetilde{u_2}$.

Still, it remains to show that $u_2\in N^{1, 1}(X)$. First, we estimate the $L^{1}$-norm of $u_2$. Since $0\leq u_2(x)\leq1$, we have
$$
\int_{X}u_2(x)d\mu\leqslant\mu(X)<\infty.
$$

Let $g_2(x):=q_2(|x|)$ in $X$. Then $g_2$ is an upper gradient of $u_2$.
It follows from Lemma \ref{lem-2.8} and the fact $0<\nu(Z)<\infty$ that
\begin{align*}
\int_{X}g_2(x)d\mu\lesssim&\int_{Z}\int_{[v_{0},\xi)}\frac{q_2(|x|)}{\nu(B_{v_x})}d\mu(x)d\nu(\xi)\\
\lesssim&\int_{0}^{\infty}q_2(t)e^{-t^{2}-\epsilon t}dt
=\sum_{n\in\mathbb{N}}\int_{n}^{n+1}\frac{2}{\int_{n}^{n+1}e^{t^{2}}dt}dt\\
\leqslant&\sum_{n\in\mathbb{N}}2e^{-n^{2}}<\infty.
\end{align*}
These ensure that $u_2\in N^{1,1}(X)$.

\subsection*{Acknowledgments}
The second author (Zhihao Xu) was partly supported by NNSFs of China under the numbers 12071121 and 12101226.

\noindent Manzi Huang,

\noindent
MOE-LCSM, School of Mathematics and Statistics, Hunan Normal University, Changsha, Hunan 410081, People's Republic of
China.

\noindent{\it E-mail address}:  \texttt{mzhuang@hunnu.edu.cn}\medskip\medskip

\noindent Zhihao Xu,

\noindent
MOE-LCSM, School of Mathematics and Statistics, Hunan Normal University, Changsha, Hunan 410081, People's Republic of
China.

\noindent{\it E-mail address}:  \texttt{734669860xzh@hunnu.edu.cn}
\end{document}